\documentclass{article}
\usepackage[utf8]{inputenc}
\usepackage{enumerate}
\usepackage{amssymb, bm}
\usepackage{amsthm}
\usepackage{amsmath}
\usepackage{tikz-cd} 
\usepackage{mathtools}
\usepackage{mathrsfs}  
 \usepackage[all]{xy}
\title{On compact K\"ahler orbifold}
\author{Xiaojun WU}
\date{\today}
\newtheorem{mythm}{Theorem}
\newtheorem{mylem}{Lemma}
\newtheorem{myprop}{Proposition}
\newtheorem{mycor}{Corollaire}
\newtheorem{mydef}{Definition}
\newtheorem{myrem}{Remark}
\newtheorem{myex}{Example}
\begin{document}
\def\cI{\mathcal{I}}
\def\Z{\mathbb{Z}}
\def\Q{\mathbb{Q}}  \def\C{\mathbb{C}}
 \def\R{\mathbb{R}}
 \def\N{\mathbb{N}}
 \def\H{\mathbb{H}}
  \def\P{\mathbb{P}}
 \def\cS{\mathcal{S}}
  \def\cC{\mathcal{C}}
  \def\cF{\mathcal{F}}
  \def\Hom{\mathrm{Hom}}
  \def\d{\partial}
 \def\dbar{{\overline{\partial}}}
\def\dzbar{{\overline{dz}}}
 \def\ii{\mathrm{i}}
  \def\d{\partial}
 \def\dbar{{\overline{\partial}}}
\def\dzbar{{\overline{dz}}}
\def \ddbar {\partial \overline{\partial}}
\def\cD{\mathcal{D}}
\def\cE{\mathcal{E}}  \def\cO{\mathcal{O}}
\def\P{\mathbb{P}}
\def\cI{\mathcal{I}}
\def \sim{\mathrm{sim}}
\def \reg{\mathrm{reg}}
\def \sing{\mathrm{sing}}
\def \an{\mathrm{an}}
\def \dR{\mathrm{dR}}
\def \top{\mathrm{top}}
\def \dim{\mathrm{dim}}
\def \Todd{\mathrm{Todd}}
\def \Spec{\mathrm{Spec}}
\def \tors{\mathrm{Tors}}
\def \det{\mathrm{det}}
\def \rk{\mathrm{rk}}
\def \ch{\mathrm{ch}}
\def \Id{\mathrm{Id}}
\def \cH{\mathscr{H}}
\def \ker{\mathrm{Ker}}
\def \rank{\mathrm{rank}}
\def \exph{\mathrm{exph}}
\def \Gr{\mathrm{Gr}}
\def \nn{\mathrm{nn}}
\def \im{\mathrm{Im}}
\def\cG{\mathcal{G}}
\def\cQ{\mathcal{Q}}
\def\SO{\mathrm{SO}}
\def\GL{\mathrm{GL}}
\def\BG{\mathrm{BG}}
\def\EG{\mathrm{EG}}
\def \Lie{\mathrm{Lie}}
\def \End{\mathrm{End}}
\def \Aut{\mathrm{Aut}}
\def \Ker{\mathrm{Ker}}
\def\g{\mathfrak{g}}
\def \red{\mathrm{red}}
\def \dim{\mathrm{dim}}
\def \Alb{\mathrm{Alb}}
\maketitle
\begin{abstract}
In this note, we study compact complex orbifolds. In the first part, we shows the equivalence of two notions of K\"ahler orbifold. 
In the second part, we shows various versions of Demailly's regularisation theorems for compact orbifold and study the positivity of orbifold vector bundle.
In the last section, we give a version of equivariant GAGA communicated to us by Brion.
\end{abstract}
\section{Introduction}
In this note, we study compact complex orbifold. Roughly speaking, a complex orbifold is a complex space which is locally a quotient of an open set of some Euclidean space under a homomorphic action of finite group.
There are two natural ways to endow a complex orbifold with a K\"ahler structure.
The first way is to define a complex orbifold to be K\"ahler if it is a K\"ahler complex space (recalled in Section 2).
The second way is to define a complex orbifold to be K\"ahler if there exist K\"ahler forms on each local (smooth) ramified cover which ``coincide" on the intersection.

From the algbraic geometric point of view, the first definition in the sense of complex space is more natural since the restriction of the Fubini-Study metric on a projective orbifold endows the orbifold a K\"ahler complex space structure.
From the analytic point of view, the second definition in the sense of orbifold is more natural since the basic tools like Sobolev embedding, elliptic regularity work identically over an orbifold as in the manifold setting.

In Section 2 of this note, we shows the equivalence of these two definitions for a compact complex orbifold.
In particular, in a fixed cohomology class, there exists a K\"ahler form in the sense of complex space if and only if there exists a K\"ahler form in the sense of orbifold.

The proof of ``if" part follows from the version of regularisation of continuous psh functions provided by Varouchas \cite{Var}.
The proof of ``only if" part follows from the partition of unity of local strictly psh functions.

We also construct the first Chern class of any coherent sheaf over a compact orbifold.
Note that in general, the Hilbert syzygy does not hold for orbifolds even if the orbifold is projective.
In particular, we cannot easily define the Chern classes of a coherent sheaf by a resolution of the coherent sheaf by locally free sheaves. 
Thus it is non-trivial to define the Chern classes of a coherent sheaf over a compact complex orbifold. 
The case of first Chern class follows from the fact that the second cohomology of the compact orbifold is isomorphic to the second cohomology of its regular part.
Without specification, we work with singular cohomology in the whole paper.

At the end of this section, using the result of \cite{Fau19}, we deduce the following version of Bogomolov inequality.
\paragraph{}

{\bf Proposition A.} {\it 
Let $E$ be an orbifold vector bundle over a compact K\"ahler complex space $(X, \omega)$ with quotient singularity.
Assume that $E$ viewing as an orbifold sheaf is $\omega-$polystable for all orbifold subsheaves.
Then the Bogomolov inequality for $E$ holds.
}
\paragraph{}
In Section 3, we verify the techniques of Demailly's regularisation for compact complex orbifolds.
Note that an orbifold current is said to be (almost) positive if it is (almost) positive over each ramified smooth cover.
Similiarly, we can define orbifold (quasi or almost) psh functions.
Note that such ramified smooth cover is unique as germs by the uniqueness of local isotropy group (recalled in Section 2).
The Lelong number and multiplier ideal sheaf of an orbifold quasi-psh function can also be defined to be corresponding notions over the ramified smooth cover.
\paragraph{}

{\bf Theorem A.} {\it 
Let $T$ be a closed almost positive $(1,1)-$orbifold current over a compact complex orbifold $X$ and
let $\alpha$ be a smooth real $(1,1)-$orbifold form in the same $\d \dbar-$cohomology class as $T$, i.e.
$T = \alpha + i \d \dbar \Psi$ where $\Psi$ is an almost psh orbifold function (i.e. almost psh on the local ramified smooth cover). Let $\gamma$ be a continuous real
$(1,1)-$orbifold form such that $T\geq \gamma$. Suppose that the orbifold tangent bundle $T_X$ is equipped with a smooth Hermitian orbifold
metric $\omega$ such that the orbifold Chern curvature form $\Theta(T_X)$ satisfies
$$(\Theta(T_X ) + u \otimes \Id_{T_X}
)(\theta \otimes \xi,\theta \otimes \xi)  \geq 0, \forall \theta, \xi \in T_X \; \mathrm{with} \;\langle \theta, \xi \rangle= 0,$$
for some continuous nonnegative $(1,1)-$orbifold form $u$ on $X$. Then there is a family of
closed almost positive (1,1)-orbifold forms $T_\varepsilon=\alpha+i \d \dbar \psi_\varepsilon$, such that $\psi_\varepsilon, \forall \varepsilon \in ]0, \varepsilon_0[$ is orbifold
smooth over $X$, increases with $\varepsilon$, and converges to $\Psi$ as $\varepsilon$ tends to 0 (in particular,
$T_\varepsilon$ is orbifold smooth and converges weakly to $T$ on $X$), and such that

(i) $T_\varepsilon \geq \gamma -\lambda_\varepsilon u-\delta_\varepsilon \omega$ where:

(ii) $\lambda_\varepsilon (x)$ is an increasing family of continuous functions on $X$ such that
$\lim_{\varepsilon \to 0} \lambda_\varepsilon (x) = \nu(T,x)$ (Lelong number of $T$ at $x$, defined on the local ramified smooth cover) at every point,

(iii) $\delta_\varepsilon$ is an increasing family of positive constants such that $\lim_{\varepsilon \to 0} \delta_\varepsilon = 0$.

On the other hand, we have the following singularity attenuation process.
For every $c > 0$, there
is a family of closed almost positive (1,1)-orbifold currents $T_{c,\varepsilon }= \alpha + i\d \dbar \psi_{c,\varepsilon}, \varepsilon \in ]0,\varepsilon_0[$,
such that $\psi_{c,\varepsilon}$ is orbifold smooth on $X \setminus E_c(T)$, increasing with respect to $\varepsilon$, and converges
to $\Psi$ as $\varepsilon$ tends to 0 (in particular, the current $T_{c,\varepsilon}$ is orbifold smooth on $X \setminus E_c(T)$ and
converges weakly to $T$ on $X$), and such that

(i) $T_{c,\varepsilon}\geq \gamma-\min(\lambda_\varepsilon,c)u -\delta_\varepsilon \omega$ where:

(ii) $\lambda_\varepsilon(x)$ is an increasing family of continuous functions on $X$ such that
$\lim_{\varepsilon \to 0} \lambda_\varepsilon (x) = \nu(T,x)$

(iii) $\delta_\varepsilon$ is an increasing family of positive constants such that $\lim_{\varepsilon \to 0} \delta_\varepsilon = 0$,

(iv) $\nu(T_{c,\varepsilon},x) = (\nu(T,x) -c)_+$ at every point $x \in X$.
}
\paragraph{}

{\bf Theorem B.} {\it 
Let $\varphi$ be an almost psh orbifold function on a compact complex orbifold
$X$ such that 
$i \d \dbar \varphi \geq \gamma$ for some continuous (1,1)-orbifold form $\gamma$. Then there is a sequence of
almost psh orbifold functions $\varphi_m$ such that $\varphi_m$ has the same singularities (on the local ramified smooth cover) as a logarithm of a sum
of squares of holomorphic functions and a decreasing sequence $\varepsilon_m > 0$ converging to 0
such that

(i) $\varphi(x) \leq \varphi_m(x) \leq \sup_{|\xi-x|<r}
\varphi(\xi) + C
(|\log r|/
m + r + \varepsilon_m
)$
with respect to ramified smooth coordinate open sets covering $X$. In particular, $\varphi_m$ converges to $\varphi$
pointwise and in $L^1(X)$ and

(ii) $\nu(\varphi,x) -\frac{n}{m} \leq \nu(\varphi_m,x) \leq \nu(\varphi,x)$ for every $x \in X$;

(iii)$i \d \dbar \varphi_m \geq \gamma -\varepsilon_m \omega$.
}
\paragraph{}

{\bf Theorem C.} {\it 
Let $\varphi$ be an almost psh orbifold function on a compact complex orbifold
$X$ such that 
$i \d \dbar \varphi \geq \gamma$ for some continuous (1,1)-orbifold form $\gamma$. Then there is a sequence of
almost psh orbifold functions $\varphi_m$ such that $\varphi=\lim_{m \to \infty}\varphi_m$
and

(i) $\varphi_m$ is orbifold smooth in the complement $X \setminus Z_m$
of an analytic set $Z_m \subset X$;

(ii) $\varphi_m$ is a decreasing sequence, and $Z
_m \subset Z_{m+1}$ for all $m$;

(iii)
$\int_X (e^{-2\varphi} - e^{-2\varphi_m} ) dV_\omega$ is finite for every $m$ and converges to 0 as $m \to \infty$;

(iv)
$\cI(\varphi_m)=\cI(\varphi)$
for all $m$ (“equisingularity”) on the local ramified smooth cover;

(v) $i \d \dbar \varphi_m \geq \gamma-\epsilon_m \omega$
where $\lim_{m \to \infty} \epsilon_m=0$.
}
\paragraph{}
In Section 4, we study the strongly pseudoeffective (strongly psef for short) orbifold vector bundle over a compact orbifold which generalises the results of \cite{Wu20}.

In the definition of the pseudoeffective vector bundle over a compact complex manifold, an additional condition is made on the approximate singular metrics on the tautological line bundle.
One may ask whether the additional condition could be made directly on some positive singular metric on the tautological line bundle (without approximation).
For example, let $h$ be a positive singular metric on $\cO_{\P(E)}(1)$, could we consider the condition that the projection of singular set $\{h = \infty\}$ in $X$ is a (complete) pluripolar set?
However, this kind of condition does not behave well functorially as shown in Example 8 suggested to the author by Demailly.

Next, we show the analogue of Theorem 1.18 \cite{DPS94} in the orbifold setting.
\paragraph{}

{\bf Theorem D.} {\it
Suppose that $X$ is a compact K\"ahler orbifold. Then an orbifold vector bundle
$E$ over $X$ is numerically flat (i.e. both $E,E^*$ are nef) if and only if $E$ admits a filtration
$$\{0\}= E_0  \subset E_1 \subset \cdots \subset E_p = E$$
by orbifold vector subbundles such that the quotients $E_k/E_{k-1}$ are orbifold Hermitian flat vector bundles.
}
\paragraph{}
The main difficulty is the following.
Given a torsion-free coherent sheaf over an irreducible compact complex space, there always exists a smooth model such that the pullback of the coherent sheaf modulo torsion is locally free.
Thus one can reduce to argue in the vector bundle case.
However, it is unclear that given a torsion-free orbifold coherent sheaf over a compact orbifold, there exists a modification as a composition of orbifold morphisms such that the pullback of the sheaf modulo torsion is an orbifold vector bundle.
Thus instead of showing all Chern class's inequalities as in Section 2 of \cite{DPS94}, we use only Bogomolov inequality shown in Proposition A.

As a geometric application, we obtain the following generalisation of Theorem 7.7 in \cite{BDPP}.
\paragraph{}

{\bf Corollary A.} {\it
For a compact K\"ahler orbifold $X$, if $c_1(X) = 0$ and $T_X$ is strongly psef, then a finite quasi-\'etale cover of $X$
is a torus.
In particular, an irreducible symplectic, or Calabi-Yau orbifold does not have a strongly psef orbifold tangent bundle or orbifold cotangent bundle.}
\paragraph{}
The proof relies on the orbifold version of the Beauville-Bogomolov decomposition theorem proven by Campana \cite{Cam04}.

\paragraph{}
Based on the orbifold version of Demailly's regularisation in Section 3, we can generalise the Serge current techniques in \cite{Wu20} to the orbifold setting.
\paragraph{}
{\bf Theorem E.} {\it 
Let $\pi: X \to Y$ be an orbifold submersion between compact K\"ahler orbifolds of relative dimension $r-1$. 
Let $T$ be a closed positive $(1,1)-$orbifold current in the cohomology class $\{ \alpha \} \in H^{1,1}(X, \R)$
such that $T$ has analytic singularities (meaning having analytic singularities in each ramified smooth cover) and is orbifold smooth on $X \setminus \pi^{-1}(Z)$ with $Z$ a closed analytic set of codimension at least $k$.
Assume that for any $y \in Y$, there exist an open neighborhood $U$ of $y$ and a quasi-psh orbifold function $\psi$ on $X$ such that
$\alpha + i \d \dbar \psi \geq 0$
in the sense of orbifold currents on
$\pi^{-1}(U)$ and $\psi$ is orbifold smooth outside a closed analytic set of codimension at least $k+r$.
Then there exists a closed positive orbifold current in the cohomology class $\pi_* \alpha^{r+k-1}$.
}
\paragraph{}
Based on an example of Zhang \cite{Zha91}, \cite{Zha93}, we show that the possible generalisation of \cite{BDPP} to the category of projective orbifolds could not hold.
More precisely, there exists a rational projective variety $X$ with quotient singularities such that the orbifold canonical line bundle is pseudoeffective.


The last section is a separate topic from the others.
Let $X$ be a projective manifold. 
Let $\cF$ be a holomorphic coherent sheaf over $X$.
Serre's GAGA implies that
$\cF$ is the analytification of some algebraic coherent sheaf.
It is a non-trivial question if there exists some algebraic action of some algebraic group $L$ on $X$ that an $L-$equivariant holomorphic vector bundle is the analytification of some $L-$equivariant algebraic vector bundle. 
Since $L$ is not compact, the usual GAGA does not apply. 
For some kind of $L$, we have the following equivariant version of GAGA communicated to us by Brion.
\paragraph{}

{\bf Proposition A. (Brion)} {\it 
 Let $X$ be a projective manifold with an (algebraic) action of a connected reductive group $L$.
Let $E$ be a holomorphic $L-$equivariant vector bundle over $X$.
Then there exists a connected finite cover $L'$ of $L$ such that $E$ is an algebraic $L'-$equivariant vector bundle.
}
\paragraph{}
This is a generalisation of Theorem 5.2.1 of \cite{Bri18} of rank one case to higher rank case.

We also have the following characterisation of projective toric varieties.
\paragraph{}

{\bf Proposition B. } {\it
Let $X$ be a connected compact K\"ahler manifold which is $(\C^*)^n-$quasi-homogeneous with $n=\dim_\C(X)$ such that the action on the open dense set is the multiplication of $(\C^*)^n$.
Then $X$ is a projective toric variety.
}
\paragraph{}
\textbf{Acknowledgement} I thank Jean-Pierre Demailly, my PhD supervisor, for his guidance, patience, and generosity. 
I would like to thank my post-doc mentor Mihai P\u{a}un for many supports.
I would like to thank Junyan Cao, Patrick Graf, Daniel Greb, Mihai P\u{a}un, Thomas Peternell, Philipp Naumann for some very useful suggestions on the previous draft of this work.
In particular, I warmly thank Michel Brion for providing me with the arguments of the last section.
This work is supported by DFG Projekt Singuläre Hermitianische Metriken für Vektorbündel und Erweiterung kanonischer Abschnitte managed by Mihai P\u{a}un.
\section{Preliminaries on orbifold}
For the simplicity of the exposition, we recall some basic definitions related to complex orbifolds, borrowed form \cite{MM}.
We first define a category $\mathcal{M}_s$ (resp. $\mathcal{M}_k$ ) as follows: the objects of $\mathcal{M}_s$ are the class of
pairs $(G, M )$ where $M$ is a connected smooth (resp. connected smooth K\"ahler) manifold and $G$ is a finite group
acting effectively on $M$ (i.e., if $g \in G$ satisfies $gx = x$ for any $x \in M$, then $g$ is
the unit element of $G$). If $(G, M )$ and $(G',M')$ are two objects, then a morphism
$\Phi:(G,M) \to (G',M')$ is a family of open holomorphic embeddings $\varphi: M \to M'$ (resp. preserving the K\"ahler form) satisfying:
\begin{enumerate}
\item For each $\varphi \in \Phi$, there is an injective group homomorphism $\lambda_\varphi: G \to G'$
that makes $\varphi$ be $\lambda_\varphi$ -equivariant.
\item For $g \in G'$ , $\varphi \in \Phi$, we define $g \varphi: M \to M'$ by $(g \varphi)(x) = g \varphi(x)$ for $x \in M $.
If $(g \varphi)(M) \cap \varphi(M) \neq \emptyset$, then $g \in \lambda_\varphi (G)$.
\item For $\varphi \in \Phi$, we have $\Phi = \{g\varphi, g \in G' \}$.
\end{enumerate}
\begin{mydef}
Let $X$ be a paracompact Hausdorff space and let $\mathcal{U}$ be a covering
of $ X$ consisting of connected open subsets in the category $\mathcal{M}_s$. We assume that $\mathcal{U}$ satisfies the condition :
For any $x \in U \cap U'$, $(U, U' \in  \mathcal{U})$, there is
$U'' \in \mathcal{U}$ such that $x \in U'' \subset U \cap U'$.

Then an orbifold structure $\mathcal{V}$ on $X$ consists in the following data:
\begin{enumerate}
\item a ramified covering $\mathcal{V}(U ) =\{(G _U , \tilde{U} \xrightarrow{\tau}
U )$, for any $U \in \mathcal{U}$, giving an
identification $U \simeq \tilde{U}/G_U\}$.
\item a morphism $\varphi_{ V U} : (G_U , \tilde{U} )\to (G_V , \tilde{V})$, for any $U, V \in \mathcal{U}, U \subset V$, which
covers the inclusion $ U \subset V$ and satisfies $\varphi_{ W U }= \varphi_{ W V} \circ \varphi_{ V U}$ for any $U, V, W \in \mathcal{U}$, with $U \subset V \subset W $.
\end{enumerate}
\end{mydef}
If $\mathcal{U}'$ is a refinement of $\mathcal{U}$ satisfying the condition in the definition, then there is an orbifold structure
$\mathcal{V}$ such that $\mathcal{V}  \cup \mathcal{V}'$ is an orbifold structure. We consider $\mathcal{V}$ and $\mathcal{V}'$ to be equivalent.
Such an equivalence class is called an orbifold structure over $X$. So we may choose
$\mathcal{U}$ arbitrarily fine.
A space with an orbifold structure is called a complex orbifold.
Note that a complex orbifold is equivalent to a complex space with only quotient singularities.

Let $(X, \mathcal{V})$ be a complex orbifold. For each $x \in X$, we can choose a small neighbourhood $(G_x, \tilde{U}_x) \to U_x$ such that $x \in \tilde{U}_x$ is a fixed point of $G_x$ (such $G_x$ is unique
up to isomorphisms for each $x \in X$ called the local isotropy group). We denote by $|G_x |$ the
cardinal of $G_x $. If $|G_x | = 1$, then $x$ is a smooth point of $X$. If $|G_x | > 1$, then $x$
is a singular point of $X$. We denote by $X_{\sing} = \{x \in X; |G_x | > 1\}$ the singular
set of $X$.
Without loss of generality, we can assume that the action of $G_x$ is linear acting on some Euclidean open set $\tilde{U}_x$.
The proof of the uniqueness of the local isotropy group can be found for example in Theorem 2 of \cite{Pri67}.

The ``orbifold" smooth form on $U_x$ is defined to be $G_x-$invariant smooth form on $\tilde{U}_x$.
An ``orbifold" K\"ahler form on $X$ is defined to be a collection of smooth forms on $U_x$ which satisfies the compatibility condition (i.e. a  complex orbifold covered by objects in the category $\mathcal{M}_k$).
By the partition of unity, it is easy to see that the de Rham cohomology group defined by the ``orbifold" smooth forms is isomorphic to the singular cohomology of the orbifold following the analogous arguments in the manifold case.
In fact, it comes from the fact that the sheaf of ``orbifold" smooth forms is soft and that the sheaves of ``orbifold" smooth forms give a resolution of the locally constant sheaf.
(In particular, for any orbifold of dimension $n$, $H^i(X,\R)=0 (\forall i > 2n)$).

Notice that $(X, \mathcal{V})$ is also a (reduced) complex space (where the germs of the holomorphic function sheaf are the holomorphic functions on the ramified smooth cover which are invariant under the local isotropy group action).
Thus the smooth forms are defined on $X$ as follows (see e.g. \cite{Dem85}).
\begin{mydef}
Since the definition of smooth forms is local in nature, without loss of generality, we can assume that $X$ is contained in an open set $\Omega \subset \C^N$.
The space of $(p,q)$-forms on $X$ (or $d$-forms on $X$) is defined to be the image of the restriction morphism
$$\cC^\infty_{p,q}(\Omega) \to \cC^\infty_{p,q}(X_\reg)~~~(\mathrm{resp.} \;\cC^\infty_{d}(\Omega) \to \cC^\infty_{d}(X_\reg)).$$
The de Rham (or Dolbeault) cohomology is defined to be the hypercohomology of the complex of sheaves of smooth forms (or $\cC^\infty_{p,\bullet,X}$ for some $p$).
\end{mydef}
Since the pullback of a smooth form via the composition of the local quotient map and the local embedding is always a smooth form invariant under the local finite group action,
a smooth form on $X$ is always an ``orbifold" smooth form.
However, in general, an ``orbifold" smooth form is not necessarily a smooth form.
A typical case is the following easy example.
\begin{myex}
{\em 
Consider $X=\C^2/\langle \pm 1\rangle$ (which embeds in $\C^3$ by sending $(u,v)$ to $(z_1,z_2,z_3):=(u^2,v^2,uv)$),
the quotient of $\C^2$ under the action of $\Z/ 2 \Z$ by changing the sign.
An ``orbifold" smooth form is to consider smooth forms on $\C^2$ that are invariant under $\Z/ 2 \Z$ (e.g. $i \d \dbar (|u|^2+|v|^2)$). 
It corresponds to the restriction of a form on $\C^3$ with conic singularities along the divisor $\{z_1z_2=0\}$ (e.g. $i \d \dbar (|z_1|+|z_2|)$ in the above example).

Notice that this example provides also an easy example where $\pi_1(X_{\reg}) \neq \pi_1(X)$.
$X_{\reg}$ admits a double cover $\C^2\setminus \{0\}$ which is simply connected.
Thus $\pi_1(X_{\reg}) =\Z/2\Z$.
On the other hand, the map $[0,1] \times \C^2 \to \C^2:(t, (u,v)) \mapsto (tu,tv)$ induces a homotopy from $X$ to a point. Thus $X$ is simply connected.
}
\end{myex}
On the other hand, we have the following elementary lemma. We refer to the note \cite{Wu21} for more information on the comparison of different cohomologies over a singular space.
\begin{mylem}
Let $X$ be a compact complex orbifold which is also a K\"ahler complex space.
The image of a K\"ahler class under the natural morphism
$$H^2_{\dR}(X, \R) \to H^2(X, \R)$$
is an ``orbifold" K\"ahler class.

In particular, a K\"ahler complex space which is an orbifold is a K\"ahler orbifold in the category $\mathcal{M}_k$
\end{mylem}
To prepare ourselves for the partition of unity type of arguments, we need to study further the behaviour with respect to a change of coordinates charts.
First, we reduce to consider only changes with small group action.
\begin{mydef}
An element
$g \in GL_n (\C)$ is called a complex reflection if the invariant space $\dim_{\C}(\C^n )^g = n - 1$. A subgroup $G \subset GL_n (\C)$
is called a complex reflection group if $G$ is generated by complex reflections. A subgroup
$G \subset GL_n (\C)$ is called a small group if $G$ does not contain complex reflections.
\end{mydef}
Recall that we have the Chevalley-Shephard-Todd theorem \cite{Che}, \cite{ST}.

\it{The quotient algebraic
variety $\C^n /G$ is smooth if and only if $G$ is a complex reflection group. In such a case
$\C^n/ G \simeq \C^n$.
}

\rm{
Note that the subgroup generated by complex reflections is normal as a matrix conjugate
of a complex reflection is a complex reflection. Thus to study quotient varieties $\C^n /G$ for $G \subset GL_n (\C)$ one
may restrict attention to quotient singularities for small groups $G$.}
In particular, in this case, the quotient map $\C^n \to \C^n/G $ is \'etale in codimension 1 (or quasi-\'etale).
(In fact, the singular locus of $\C^n/G$ is the image of $\cup_{g \in G, g \neq e} (\C^n)^g$ under the quotient map.)

Let $(U_\alpha)_\alpha$ be a finite open cover of $X$ such that $U_\alpha \simeq \tilde{U}_\alpha/ G_\alpha$ where $\tilde{U}_\alpha$ are euclidean open sets and $G_\alpha \subset GL_n (\C)$ are small.
Without loss of generality, we can assume that $U_\alpha$ can be embedded as a closed analytic subset of some open set $V_\alpha$ of $\C^{N_\alpha}$ for certain $N_\alpha$.
Then $\omega|_{U_\alpha}$ is the pull back of some smooth form $\omega_\alpha$ defined on $V_\alpha \subset \C^{N_\alpha}$.
Now we choose a ``good" chart of $U_\alpha \cap U_\beta$. 
\begin{myprop}
There exists a smooth manifold $\tilde{V}_{\alpha, \beta}$ such that $U_\alpha \cap U_\beta \simeq \tilde{V}_{\alpha, \beta}/ (G_\alpha \times G_\beta)$ and a commuting diagram
$$\begin{tikzcd}
\tilde{U}_\alpha \arrow[d] & {\tilde{V}_{\alpha,\beta}} \arrow[d] \arrow[l] \arrow[r] & \tilde{U}_\beta \arrow[d] \\
U_\alpha                   & U_\alpha \cap U_\beta \arrow[l] \arrow[r]                & U_\beta                  
\end{tikzcd}$$
such that $\tilde{V}_{\alpha, \beta} \to \tilde{U}_\alpha$ is $G_\alpha-$equivariant and $G_\beta-$principal bundle over its image and similarly for $\tilde{V}_{\alpha, \beta} \to \tilde{U}_\beta$.
\end{myprop}
\begin{proof}
Define
$$\tilde{V}_{\alpha, \beta}\coloneqq \{(z_\alpha,z_\beta) \in \tilde{U}_\alpha \times \tilde{U}_\beta, \pi_\alpha(z_\alpha)=\pi_\beta(z_\beta)\}$$
with natural induced $G_\alpha \times G_\beta-$action.
By Definition 1, $\tilde{V}_{\alpha, \beta}$ is smooth.
Note that $\tilde{V}_{\alpha, \beta} \to \tilde{U_i}$ is $G_i-$equivariant for $i=\alpha, \beta$.

On the other hand, $\tilde{V}_{\alpha, \beta} \to \tilde{U_i}$ is $G_j-$unramified over a Zariski open set outside an analytic set of codimension at least 2 for $\{i,j\}=\{\alpha, \beta \}$.
Since $\tilde{V}_{\alpha, \beta}$ is normal and $\tilde{U}_i$ is smooth,
by the purity of Zariski-Nagata (cf. e.g. P170 \cite{Fis76}),
$\tilde{V}_{\alpha, \beta} \to \tilde{U_i}$ is $G_j-$unramified onto the image.
\end{proof}
\begin{proof}(of Lemma 1)

Let $(\theta_\alpha)_\alpha$ be a partition of unity associated with this open cover in the above Proposition 1.
More precisely, $(\theta_\alpha)_\alpha$ are constructed as follows.
Let $U'_\alpha \subset \subset U_\alpha$ such that $U'_\alpha$ still cover $X$.
Let $\theta_\alpha$ be smooth function on $\C^{N_\alpha}$ such that $\theta_\alpha \equiv 1$ on $U'_\alpha$  and $\theta_\alpha$  is compactly supported in $V_\alpha$. 
Denote $i_\alpha: U_\alpha \to V_\alpha \subset \C^{N_\alpha}$ the closed immersion and $\pi_\alpha: \tilde{U}_\alpha \to U_\alpha$ the quotients.
Notice that $\sum_\alpha \theta_\alpha >0$ on $X$.

Let $(z_{i, \alpha})$ be the local coordinates of $\tilde{U}_\alpha$.
Consider the set $\{g_{i, \alpha}\}$ generated by $g.z_{i, \alpha}$ for all $g \in G_\alpha$ and $i$ which is thus $G_\alpha-$invariant.
In particular, 
$$\psi_\alpha \coloneqq \sum_i |g_{i, \alpha}|^2 $$
is invariant under $G_\alpha$ and descends to a function on $U_\alpha$.
Notice that $i \d \dbar \psi_\alpha$ is strictly positive on $\tilde{U}_\alpha$.
With the same notations in the above proposition, the pullback of
$\psi_\alpha$ defines a $G_\alpha-$invariant function on $\tilde{V}_{\alpha,\beta}$.
Since $\tilde{V}_{\alpha, \beta} \to \tilde{U}_\beta$ is $G_\alpha-$principal bundle over its image,
the pullback of
$\psi_\alpha$ descends to a smooth function on $\pi_\beta^{-1}(U_\alpha \cap U_\beta)$ for any $\beta$ under identification.
In particular, $ \psi_\alpha$  also defines a smooth function on $\pi_\beta^{-1}(U_\alpha \cap U_\beta)$ such that $i \d \dbar \psi_\alpha$ is strictly positive on $\pi_\beta^{-1}(U_\alpha \cap U_\beta)$.

We claim that for $\varepsilon >0$ small enough, 
$$\omega+\varepsilon i \d \dbar (\sum_\alpha i_\alpha^* \theta_\alpha \psi_\alpha)$$
defines an orbifold K\"ahler form over $X$ which finishes the proof.

Let $\omega$ be a K\"ahler form on $X$.
For fixed $\alpha$, fix a smooth Hermitian metric on $T_{\tilde{U_\alpha}}$.
$$\pi^*_\alpha[\omega+\varepsilon i \d \dbar (\sum_\beta i_\beta^* \theta_\beta \psi_\beta)]=\pi_\alpha^*[i_\alpha^* \omega_\alpha+\varepsilon \sum_\beta(\psi_\beta i \d \dbar i_\beta^* \theta_\beta
+i \d \psi_\beta \wedge \dbar i_\beta^* \theta_\beta 
-i \dbar \psi_\beta \wedge \d  i_\beta^* \theta_\beta
+i_\beta^* \theta_\beta i \d \dbar \psi_\beta )].$$
Define closed analytic set $A \coloneqq \{X \in T_z \tilde{U}_\alpha; D_z(i_\alpha \circ \pi_\alpha)(X)=0 \}$.
By the above proposition, under identification, on $T(\tilde{U_\alpha} \cap \pi_\alpha^{-1}(U_\alpha \cap U_\beta))$, 
$D(i_\beta \circ \pi_\beta)(A)=0$ for any $\beta$.
To prove that $\pi^*_\alpha[\omega+\varepsilon i \d \dbar (\sum_\beta i_\beta^* \theta_\beta \psi_\beta)]$
is a K\"ahler form over $\tilde{U}_\alpha$, 
 it is enough to show that 
for $\varepsilon >0$ small enough, 
$F_\alpha(X)=\pi^*_\alpha[\omega+\varepsilon i \d \dbar (\sum_\beta i_\beta^* \theta_\beta \psi_\beta)](X,X)$
defines a strictly positive function for $X$ in the unit sphere bundle over $\tilde{U}_\alpha$ (with respect to the local Euclidean metric or any orbifold smooth metric on the orbifold tangent bundle). 
Since $\omega$ is a K\"ahler form,
$\pi^*_\alpha \omega (X,X)$ is strictly positive outside $A$ and positive over the unit sphere bundle.
Since $\sum_\alpha \theta_\alpha >0$ on $X$ and $i \d \dbar \psi_\beta$ is strictly positive on $\pi_\alpha^{-1}(U_\alpha \cap U_\beta)$,
$$\pi^*_\alpha (\sum_\beta i_\beta^* \theta_\beta i \d \dbar \psi_\beta)(X,X)$$
is strictly positive on the unit sphere bundle.
On the other hand,
$$\pi^*_\alpha (\sum_\beta \psi_\beta i \d \dbar i_\beta^* \theta_\beta
+i \d \psi_\beta \wedge \dbar i_\beta^* \theta_\beta 
-i \dbar \psi_\beta \wedge \d  i_\beta^* \theta_\beta)$$
vanishes on $A$ since all terms involve the derivatives of $\pi^*_\alpha i_\beta^* \theta_\beta$.
Without loss of generality, we can assume that all forms are defined over the closure of $\tilde{U}_\alpha$.
Since the unit sphere bundle in $T_{\overline{\tilde{U}_\alpha}}$ is compact,
for $\varepsilon >0$ small enough,
$F_\alpha$ is strictly positive. 
\end{proof}
Now we give the definition of the first Chern class of a coherent sheaf (not necessarily torsion-free) over a compact complex orbifold.
To do it, we need a more cohomological study of compact complex orbifold.
\begin{mylem} {\rm(analogue of lemma 11.13 in  \cite{Voi})}

Let $X$ be a complex orbifold (not necessarily compact) of dimension $n$ and $Y$ be a closed analytic subset of codimension at least $r+1$.
Then the restriction map
$$H^l(X, \R) \to H^l(X \setminus Y,\R)$$
is an isomorphism for $l \leq 2r$.
The same holds if changing $\R$ by $\Q$.
\end{mylem}
\begin{proof}
In the following, we only show the real case.
The proof of the rational case is identical.
 Notice that $X \setminus Y$ is also a complex orbifold.
 By Poincaré-Verdier duality, it suffices to prove that
 the inclusion map
 $$H^l_c(X \setminus Y, \R) \to H^l_c(X ,\R)$$
is an isomorphism for $l \geq 2n-2r$.
We have the long exact sequence for cohomology with compact supports
$$\cdots H^{l}_c( X \setminus Y, \R) \to H^l_c(X, \R) \to H^l_c(Y\, \R) \to H^{l+1}_c( X\setminus Y, \R)\cdots .$$
It reduces to show that
for $l \geq 2n-2r$,
$H^l_c(Y, \R)=0$ which is true by dimension condition.
\end{proof}
In particular, if $X$ is a (compact) complex orbifold, $X_\sing$ is a closed analytic of codimension at least 2, which implies that
$$H^2(X,\R) \simeq H^2(X_\reg,\R) (\mathrm{resp.} \; H^2(X,\Q) \simeq H^2(X_\reg,\Q)).$$
Notice that the analogue in integral coefficients is in general false. For example, consider $X= \C^2/\langle \pm 1 \rangle$ which is contractible and hence has trivial cohomology other than degree 0 case.
However, its regular part is homotopic to $S^3/ \langle \pm 1 \rangle \simeq \R\P^2$ which has non-trivial torsion cohomology groups.
\begin{mydef}
Let $\cF$ be a coherent sheaf over a compact complex orbifold $X$.
We define the first Chern class of $\cF$ to be the preimage of $c_1(\cF|_{X_{\reg}})$ in the above isomorphism.
The construction of the Chern class of a coherent sheaf over the non-compact manifold $X_{\reg}$ is recalled in the following Remark 1.
\end{mydef}
Let us recall some results on coherent real analytic space to define the Chern classes of coherent sheaves over open manifolds.
First, we recall the following characterisation when the underlying real analytic space of a complex space is coherent. 
\begin{myprop}(Chap II. Proposition 2.15 \cite{BMT86})

Let $X$ be a (reduced) complex analytic space and
$X^\R$ be its underlying (reduced) real analytic space. 
$X^\R$
is coherent at point $a$ if and only if the irreducible components of $X_a$ remain
irreducible in a neighbourhood of $a$.
\end{myprop}
In particular, the underlying space of a locally irreducible complex space is coherent.
The advantage of coherent real analytic spaces is that they are similar to Stein spaces.
For example, we have the following real analytic analogue of Cartan theorems A and B.
\begin{mythm}((Chap II. Theorem 3.7 \cite{BMT86}))
 
Let $(X, \cO_X^{\R-\an})$ be a reduced coherent real analytic space and $\cF$ be a coherent $\cO_X^{\R-\an}-$sheaf.

(1) For any $x \in X$, $\cF_x$ is generated as $\cO_{X,x}^{\R-\an}-$module by $H^0(X, \cF)$.

(2) $H^q(X, \cF)=0$ for any $q >0$.  
\end{mythm}
\begin{myrem}
{\em 
If $X$ is a compact complex manifold of complex dimension $n$, by the above version of Cartan Theorem A, there exists a surjective morphism
$\cO_X^{\R-\an, N_0}\coloneqq E_0 \to \cF \otimes_{\cO_X} \cO_X^{\R-\an} $ for some $N_0 \in \N$.
The kernel is a coherent $\cO_X^{\R-\an}-$sheaf to which we have a surjective morphism from some  $\cO_X^{\R-\an, N_1}\coloneqq E_1$.
Let us continue this process and consider the kernel of $E_1 \to E_0$.
By Hilbert's syzygy theorem, the kernel of $E_{2n-1}\to E_{2n-2}$ has to be locally free, and we denote it by $E_{2n}$.
In conclusion, $\cF\otimes_{\cO_X} \cO_X^{\R-\an}$ admits a resolution of real analytic vector bundles $E_\bullet$ over $X$ (in fact over any compact coherent real analytic manifold).

If the complex space $X$ is not compact (e.g. $X$ is the regular part of some compact complex orbifold), 
we can define the Chern classes of a coherent sheaf in the following way.
Assume that $X$ is some open set of a compact irreducible complex space $Y$ which is smooth. 
The space $Y$ can be covered by a finite closed analytic subset of some euclidean open set.
The same holds for the open set $X$.
Thus it is second countable since any closed analytic subset of some euclidean open set is second countable.
It is paracompact since it is locally compact and Hausdorff.
Let $(X_\nu)$ be an increasing sequence of open subsets of $X$ such that $X= \bigcup_\nu X_\nu$ and $X_\nu$ is relatively compact in $X_{\nu+1}$.
Then by the above version of Cartan Theorem A, there exist finite sections of $H^0(X_{\nu+1}, \cF)$ which generate $\cF_x \otimes_{\cO_{X,x}} \cO_{X,x}^{\R-\an}$ at any point $x \in X_\nu$.
Continue the process of the construction of the resolution of the coherent sheaf $\cF \otimes_{\cO_X} \cO_X^{\R-\an}|_{X_\nu}$ as above.
The process will terminate in a finite number of steps and produce in this way a resolution $E^\bullet_\nu$ of $\cF\otimes_{\cO_X} \cO_X^{\R-\an} |_{X_\nu}$.
As above define
$$\ch(\cF|_{X_\nu}) \coloneqq \sum_i (-1)^i \ch(E^i_\nu).$$
On the other hand, the singular cohomology group $H^\bullet(X, \R)$
satisfies
$$H^\bullet(X, \R)=\varprojlim_\nu H^\bullet(X_\nu, \R).$$
(The reason is as follows due to David E. Speyer. Since the image of any singular chain is compact, it is contained in some $X_\nu$ for $\nu$ sufficiently large.
Thus the singular chain complex satisfies
$$C_\bullet(X, \R)=\varinjlim_\nu C_\bullet(X_\nu, \R).$$
Since direct limit functor is exact, we have the isomorphism in homology.
By universal coefficients theorem, for any $\nu$,
$H^\bullet(X_\nu, \R)\simeq \mathrm{Hom}(H_\bullet(X_\nu,\R),\R)$.
The inverse limit system $(H^\bullet(X_\nu, \R))_\nu$ satisfies the Mittag-Leffler condition, that is for any $\mu > \nu$, $H^\bullet(X_\mu, \R) \to H^\bullet(X_\nu, \R)$ has finite dimensional image and thus we have isomorphisms in cohomologies.
Notice that that $\mu > \nu$, $H^\bullet(X_\mu, \R) \to H^\bullet(X_\nu, \R)$ factorises through $H^\bullet(\overline{ X}_\nu, \R)$ which is finite dimensional.) 

Since for any $\nu$, the definition of $\ch(\cF|_{X_\nu})$ is independent of the choice of the resolution,
$(\ch(\cF|_{X_\nu}))_\nu$ defines an element in $\varprojlim_\nu H^\bullet(X_\nu, \R)$. 
We define this element as the Chern characteristic class of $\cF$ and denote it by $\ch(\cF)$.
The same arguments work in the rational coefficients case.
}
\end{myrem}
Notice that the orbifold assumption in the Lemma 2 is necessary as shown by the following example.
\begin{myex}
{\em
(Cohomology of cone)

Let $X$ be a projective manifold and $A$ be a very ample line bundle over $X$.
We denote the total space of the dual line bundle as $A^{-1}$.
The affine cone $C(X)$ of $X$ with respect to $A$ is defined to be $\Spec (\oplus_{i \geq 0} H^0(X, A^i))$.
It has only one singular point at the vertex.
The blow-up of the vertex as a closed point gives the resolution of singularities of $C(X)$ which is isomorphic to $\pi:A^{-1} \to C(X)$.
The Stein factorisation of $\pi$ shows that $C(X)$ is normal.
For any $i>0$,
$R^i \pi_* (\cO_{A^{-1}})=(H^i(A^{-1},\cO_{A^{-1}}))^{\tilde{}}$ since $C(X)$ is affine.
Since the natural map $f: A^{-1} \to X$ is affine, it is also equal to
$$(H^i(X,f_* \cO_{A^{-1}}))^{\tilde{}} \simeq (\oplus_{j \geq 0} H^i(X,
A^j))^{\tilde{}}.$$
In particular, if $h^{1,0}(X) \neq 0$, $C(X)$ is not of rational singularity which is in particular not a complex orbifold.
Notice that by \cite{Vie77}, quotient singularity is always of rational singularity.

Let $PC(X)$ be the projective cone over $X$ with respect to $A$ which contains $C(X)$ as a Zariski open set.
The regular part of $PC(X)$ is homeomorphic to a rank 2 real vector bundle over $X$ which implies that
$$H^2(PC(X)_{\reg}, \R) \simeq H^2(X, \R).$$
Since $X \times \{0\}$ is a deformation retract of some neighborhood in $X \times \P^1$, by Chap. 2 Theorem 2.13 of \cite{Hat},
we have the following exact sequence for any $l>0$,
$$\cdots H^{l-1}(X \times \{0\}, \R) \to H^l(X \times \P^1/X \times \{0\}, \R) \to H^l(X \times \P^1, \R) \to H^{l}(X \times \{0\}, \R)\cdots .$$
Notice that $X \times \P^1/X \times \{0\}$ is homeomorphic to $PC(X)$.
By Kunneth formula and chasing the diagram,
we have that
$H^2(PC(X), \R) \simeq \R$.

Let $X$ be an abelian variety of dimension at least 2, the above calculations show that the second cohomology of $PC(X)$ and its regular part have different dimensions.
In particular, the restriction map does not induce an isomorphism on cohomologies.

(One can also show that $H^4(PC(X),\R) \simeq H^2(X, \R) $. 
Let $X$ be an abelian surface.
Then $PC(X)$ gives an example of singular variety not satisfying the Poincaré duality.
)
}
\end{myex}
In a bit more general case, we can define the homology first Chern class as follows.
\begin{myprop}
Let $\cF$ be a coherent sheaf over $X$ a compact irreducible complex space smooth in codimension 1.
Let $Z$ be the singular part of $X$.
Then as above, we can define the homology first Chern class of $\cF$ via the isomorphisms
$$H_{2n-2}(X, \R) \simeq H^{2n-2}(X , \R)^* \simeq H_c^{2n-2}(X \setminus Z, \R)^* \simeq H^{2}(X \setminus Z, \R).$$
If $\cF$ is torsion and $\omega$ is a K\"ahler form on $X$, then we have
$$\langle c_1^h(\cF), \omega^{n-1} \rangle \geq 0$$
where $c_1^h(\cF)$ denotes the homology first Chern class.
\end{myprop}
\begin{proof}
Consider the following commuting diagram
$$
\begin{tikzcd}
{H^{2n-2}_{c,\cD-dR}(X \setminus Z ,\R )} \arrow[r,"\simeq"] \arrow[d] & {H^{2n-2}_c(X \setminus Z,\R)} \arrow[d, "\simeq"] \\
{H^{2n-2}_{\cD-dR}(X,\R)} \arrow[r]                        & {H^{2n-2}(X ,\R)}     .
\end{tikzcd}$$
Notice that the isomorphism on the first line follows from the assumption that $X \setminus Z$ is smooth (cf. Chap IV. (7.10) \cite{agbook}). 
By \cite{Her67}, the bottom morphism is surjective.
By the diagram, it is in fact an isomorphism.

Some explanations are needed for the left arrow.
Let $\alpha$ be a smooth form with compact support on $X \setminus Z$.
Then the extension by 0 of $\alpha$ defines a smooth form on $X$.
In fact, the support of $\alpha$ is away from the singular part of $X$.
Thus locally near the singular part, the extension by 0 of $\alpha$ is just the restriction of the zero form.

The determinant line bundle of torsion sheaf $\cF|_{X \setminus Z}$ has a global non-trivial section which defines a current $[D|_{X \setminus Z}]$ representing an element in ${H^{2n-2}_c(X \setminus Z,\R)}^* \simeq {H^{2n-2}_{c,\cD-dR}(X \setminus Z ,\R )}^*$.
The closure of the divisor defined by this non-trivial section in $X$ will be denoted by $D$ which as a current defines an element in ${H^{2n-2}_{\cD-dR}(X,\R)}^* \simeq {H^{2n-2}_{}(X,\R)}^*$ whose image in ${H^{2n-2}_c(X \setminus Z,\R)}^*$ is the class of $[D|_{X \setminus Z}]$.

In particular, we have
$$\langle c_1^h(\cF), \omega^{n-1} \rangle=  \int_D \omega|_D^{n-1} \geq 0.$$
\end{proof}

In the following paragraph, we recall the orbifold Chern classes of an orbifold vector bundle 
in \cite{LT16}.
In particular, we emphasize that the orbifold Chern classes can be defined without using metrics and can be defined in the cohomologies with rational coefficients which seems to be lack of reference following \cite{LT16}.
Recall that the orbifold Chern classes of an orbifold vector bundle (recalled below) in de Rham cohomology are defined to be the classes represented by the Chern curvature forms of some orbifold smooth metric in real coefficients case.

Notice first by the construction of Definition 4, the first orbifold Chern class of an orbifold line bundle in de Rham cohomology coincides with the one in Definition 4.

Recall the following version of the Leray-Hirsch theorem due to \cite{PS03}.
\begin{mydef}
 A continuous map $p : M \to B$ is a locally trivial fibration, say with
fibre $F$, in the orbifold sense if for every $b \in B$ there exists a neighbourhood $V_b$, a
topological space $U_b$, and a topological group $G_b$ such that
\begin{enumerate}
 \item $G_b$ acts on $U_b$ and on $F$; the action on $F$ is by homeomorphisms homotopic
to the identity;
\item  $V_b$ is homeomorphic to $U_b /G_b$;
\item $p^{-1} (V_b)$ is homeomorphic to the quotient of $U_b \times F$ by the product action of $G_b$.
\end{enumerate}
 In this setting, composing the natural quotient map $F \to F/G_b$ with the homeomorphism $F/G_b \simeq p^{-1}(b)$ and the inclusion $p^{-1} (b) \to X$, defines the orbifold fibre inclusion $r_b : F \to X$.
\end{mydef}
\begin{mythm}(\cite{PS03})
Let $p : M \to B$ be a fibration which is locally trivial in the orbifold sense. Suppose that for all $q \geq 0$ there exist classes $e_1, \cdots, e_{n_q} \in  H^q (M , \Q)$
that restrict to a basis for $H^q (F , \Q)$ under the map induced by the orbifold fibre inclusion $r_b : F \to M$. The map $a \otimes r_b^*(e_i ) \mapsto p^* a \cup e_i, a \in H^\bullet(B, \Q)$ extends
linearly to a graded linear isomorphism
$$H^\bullet (B, \Q) \otimes H^\bullet (F, \Q) \simeq H^{\bullet} (M , \Q).$$
\end{mythm}
Let $\cF$ be a reflexive sheaf over a complex orbifold $X$.
By the diagram in Proposition 1, with the same notations, $(\pi_\alpha^* \cF/\tors)_\alpha$ naturally defines an orbifold sheaf on $X$  or equivariant coherent sheaves on $\tilde{U}_\alpha$ in the terminology of Geometric Invariant Theory ``gluing'' via the diagram in Proposition 1, since taking pullback modulo torsion is a functor.
We refer to \cite{Fau19} for further discussion.
To simplify the notation, we sometimes call $\{\tilde{U}_\alpha\}_{\alpha}$ local ramified smooth covers.
Here $\tors$ means the torsion part of the corresponding coherent sheaf.
The natural morphism
$ \cF \to (\pi_{\alpha,*} \pi^*_\alpha \cF/\tors)^{G_\alpha}$ is isomorphic in codimension 1.
Since $\cF$ is reflexive, it is in fact an isomorphism.
Recall that the first Chern class of an orbifold sheaf $(\cG_\alpha)_\alpha$ is defined to be the first Chern class of its determinant orbifold line bundle (i.e. the determinant line bundle of $\cG_\alpha$ on $\tilde{U}_\alpha$).
In fact, the first Chern class of $(\pi_\alpha^* \cF/\tors)_\alpha$ is equal to the first Chern class of $\cF$ by restricting on the regular part of $X$.

In the following, a reflexive sheaf $\cF$ will be called an orbifold vector bundle if for any $\alpha$, $((\pi_\alpha^* \cF/\tors)^{**})_\alpha$ is locally free.
Note that for an orbifold vector bundle $E$, the projectification $\P(E)$ of $E$ is naturally defined as an orbifold with a tautological orbifold line bundle $\cO_{\P(E)}(1)$ and natrual projection $p: \P(E) \to X$ which is an orbifold morphism.
With a possible restriction to a smaller $\tilde{U}_\alpha$, we can assume that the action on some (holomorphic) localisation of  $((\pi_\alpha^* \cF/\tors)^{**})_\alpha$ is the product of the group action on $\tilde{U}_\alpha$ and some group action on $\C^{r}$ where $r$ is the rank.
Using $G_\alpha$-invariant metrics on $((\pi_\alpha^* \cF/\tors)^{**})_\alpha$ which is compatible with the orbifold structure, one can define the positivity of an orbifold vector bundle.
\begin{mydef}
Let $(X, \omega)$ be a compact K\"ahler orbifold.
An orbifold vector bundle $E$ is said to be nef (or strongly pseudo-effective) if for any $\epsilon >0$, there exists an orbifold smooth (or singular) metric $h_\epsilon$ on $E$ such that $c_1(\cO_{\P(E)}(1), h_\epsilon)+\epsilon p^* \omega$ is positive (or positive in the sense of currents with analytic singularities such that the projection of the singular part under $p$ is not dominant over $X$) over each local ramified smooth cover using the usual formula of Chern curvature.

A numerically flat orbifold vector bundle $E$ is an orbifold vector bundle such that $E, E^*$ are nef orbifold vector bundles.
An orbifold vector bundle is said to be Hermitian flat if there exists an orbifold metric on it which is Hermtian flat over each local ramified smooth cover using the usual formula of Chern curvature.
\end{mydef}
Note that an orbifold vector bundle of rank $r$, $p:E \to X$ or its projectification $\pi: \P(E) \to X$ is a locally trivial fibration with fibre $\C^n$ (or $\P^{r-1}$) in the orbifold sense.
Notice that $\mathrm{ GL}(r,\C),\mathrm{ PGL}(r,\C)$ are connected and thus the action on the orbifold fibre is homotopic to the identity.

\begin{mydef}(Orbifold Chern classes in rational coefficients)

Let $p:E \to X$ be an orbifold vector bundle of rank $r$.
Note that $\P(E)$ is also a complex orbifold and $\cO_{\P(E)}(1)$ is an orbifold vector bundle such that the restriction of $c_1(\cO_{\P(E)}(1))$ generates $H^{\bullet}(\P^{r-1}, \Q)$ by the orbifold fibre inclusion.
By construction,  the restriction of $\cO_{\P(E)}(1)$ by the orbifold fibre inclusion is the tautological line bundle.
In fact, the restriction of $c_1(\cO_{\P(E)}(1)) \in H^2(\P(E), \R)$ by the orbifold fibre inclusion is represented by Chern curvature form whose restriction by the orbifold fibre inclusion is just the Fubini-Study metric.

Thus
there are unique elements $c_i(E) \in H^2(X, \Q)$, such that
$$c_1(\cO_{\P(E)}(1))^r +\sum_i (-1)^i \pi^* (c_i(E)) \cup c_1(\cO_{\P(E)}(1))^{r-i}=0$$
by Leray-Hirsch theorem (Theorem 2). We define the orbifold Chern classes of the orbifold vector bundle $E$ to be precisely the $c_i(E)$.

\end{mydef}
By the Leray-Hirsch theorem, it is easy to see that if an orbifold vector bundle has a nowhere vanishing section, its top orbifold Chern class is trivial.

It is interesting to define orbifold Chern classes of an orbifold vector bundle in integral singular cohomology.
One difficulty comes from the fact that for a compact complex manifold $X$ with holomorphic actions of $G$ a finite group
$$H^\bullet(X/G, \Q)=H^\bullet(X , \Q)^G.$$
However, the next example shows that the analogue in integral singular cohomologies is false in general.
An easier smooth example is the following. 
An Enriques surface $X$ is a quotient of a K3 surface by a fixed-point-free involution.
$H^2(X,\Z)$ contains a non-trivial torsion element while the invariant part of the second cohomology of the K3 surface is free.

\begin{myex}
{\rm (cohomology of singular Kummer surface due to Torsten Ekedahl)

Let $X$ be the quotient of a torus $A$ by an involution $\sigma$. 
Let $\pi: \tilde{X} \to X$ be the minimal resolution of $X$ by blowing up the 16 singular points $x_i(1 \leq  i\leq 16)$.
By proper base change, we thus have
$$\pi_* \Z_{\tilde{X}}=\Z_X, R^2 \pi_* \Z_{\tilde{X}}=\oplus_i \Z_{x_i}, R^j \pi_* \Z_{\tilde{X}}=0(\forall j \neq 0,2).$$

By Leray spectral sequence, we have 
$$H^i(X,\Z)\simeq H^i(\tilde{X},\Z) (\forall i \neq 2,3).$$
Since $\tilde{X}$ is a K3 surface, all its cohomologies are known.
Since $H^3(\tilde{X},\Z)=0$, $E^{3,0}_\infty=0$ and thus $d_3: E^{0,2}_2=H^0(X, R^2 \pi_* \Z_{\tilde{X}}) \to E^{3,0}_2=H^3(X,\Z)$ is surjective.
We have an exact sequence
$$0 \to E^{0,2}_\infty \to E^{0,2}_2 \to E^{3,0}_2 \to 0.$$
On the other hand, $E^{1,1}_\infty=E^{1,1}_2=0, E^{2,0}_\infty=E^{2,0}_2$.
As graded modules, we have an exact sequence
$$0 \to E^{2,0}_2=H^2(X,\Z) \to H^2(\tilde{X},\Z) \to E^{0,2}_\infty \to 0.$$
Thus $H^3(X,\Z)$ is the coimage of edge homomorphism $H^2(\tilde{X},\Z) \to H^0(X,  R^2 \pi_* \Z_{\tilde{X}})$.
We claim that the edge homomorphism is given by taking intersection number with the $(-2)-$curves in $\tilde{X}$.
By the functionality, it is enough to show that the edge homomorphism $H^2(\tilde{X},\C) \to H^0(X,  R^2 \pi_* \C_{\tilde{X}})$ is given by taking intersection number with the $(-2)-$curves in $\tilde{X}$.
Since $\tilde{X}$ is smooth,
$\C_{\tilde{X}}$ is quasi-isomorphic to the complex of smooth forms $A^\bullet_{\tilde{X}}$.
The edge morphism is given by $H^2(\tilde{X},\C)=\H^2(\tilde{X},A^\bullet_{\tilde{X}}) \to H^0(X, \mathcal{ H}^2 (\pi_* A^\bullet_{\tilde{X}}))$.
For any $i$, the germ $\mathcal{H}^2 (\pi_* A^\bullet_{\tilde{X}})_{x_i}$ is identified to $\C$ by integration of the smooth form representatives of $\H^2(\tilde{X},A^\bullet_{\tilde{X}})$ along the curves which finishes the proof of the claim.

By exact sequence
$$0 \to H^2(X,\Z) \to H^2(\tilde{X},\Z) \to H^0(X,  R^2 \pi_* \Z_{\tilde{X}}),$$
$H^2(X,\Z)$ is the othogonal complement of the $(-2)-$curves which identifies to $H^2(A,\Z)$ by Corollary 5.6 Chap. VIII \cite{BHPV04}. 

Since the $(-2)-$curves define classes in $H^2(\tilde{X},\Z)$, by considering their image,
the coimgae of edge homomorphism is a quotient of $(\Z/2\Z)^{\oplus 16}$.
By Proposition 5.5 Chap. VIII \cite{BHPV04},
image of induced $H^2(\tilde{X},\Z) \to (\Z/2\Z)^{\oplus 16}$ is caradinal $2^5$.
In particular,
$H^3(X, \Z)$ has non-trivial torsion element.

Note that $H^\bullet(A, \Z)^{\Z/2\Z} $
has however no torsion elements.
}
\end{myex}
There are many other examples provided in \cite{BCGP}.
\begin{myex}
{\rm (\cite{BCGP})
A surface $S$ with rational double points which is the
quotient of a product of curves by the diagonal action of a finite group is called a product-quotient surface.
In Section 5 \cite{BCGP}, they classified such surfaces with genus of the curves at least 2, $h^{0,1}(S)=h^{0,2}(S)=0$ which is singular.
In particular, the fundamental group of $S$ is finite.
Their calculations show that $H_1(S,\Z)$ is non trivial (in other words, the fundamental group is not perfect).
By universal coefficients theorem, $H^2(S,\Z)$ has non-trivial torsion element.
Thus the exact sequence $H^1(S, \cO^*_S) \to H^2(S, \Z) \to H^2(S, \cO_S)=0$ implies the existence of line bundle over $S$ with non-trivial torsion Chern class.
}
\end{myex}
\begin{myex}
{\em 
Let $X=A/ \langle \pm 1 \rangle$ as in Example 1.
Let $\pi: A \to X$ be the quotient map which can also be viewed as an orbifold morphism.
The natural morphism $T_A \to \pi^* T_X$ is in fact an isomorphism by construction.
In particular, the ``orbifold" second Chern class of $X$ satisfies 
$$c_2(X) \cap [X]=\frac{1}{2} \pi_* (\pi^* c_2(X) \cap [A])=\frac{1}{2} \pi_* (c_2(T_A)\cap [A])=0.$$
For any de Rham cohomology class (of ''orbifold" smooth forms) $\alpha$ on $X$, we have that
$$\langle \pi_* (\pi^* c_2(X) \cap [A]), \alpha \rangle=\langle (\pi^* c_2(X) \cap [A]), \pi^* \alpha \rangle=\langle [A], \pi^*(c_2(X) \wedge \alpha) \rangle$$
by the natural pair between singular cohomology and singular homology which is also equal to
$$\int_A \pi^*(c_2(X) \wedge \alpha)=2 \int_{X}c_2(X) \wedge \alpha=\langle 2 c_2(X) \cap [X], \alpha \rangle.$$
Since 
the Poincaré-Verdier duality holds on complex orbifold,
$c_2(X)=0.$ 

Another way to prove this fact is to notice that there exists an invariant flat metric on $A$ which induces a flat metric on $X$.
Moreover, the above calculations also work for $X=Y/G$ where $Y$ is a complex manifold with a holomorphic finite group $G$ action. Let $|G|$ be the order of $G$ and $\pi: Y \to X$ be the quotient map. We have that
$$c_k(X) \cap [X]=\frac{1}{|G|} \pi_*(c_k(Y) \cap [Y])$$
for any $k$.  
}
\end{myex}

Now we begin to show that any K\"aher orbifold is also a K\"ahler complex space.
For it, we need some results on pluripotential theory on complex space. We recommend the article \cite{Dem85} for further information and reference.

Recall that the definitions of test functions and currents are local in nature.
To define them on a complex space $X$,
we can identify $X$ as a closed analytic subset of an open set $\Omega \subset \C^N$ as recalled before.
The topology of smooth forms is induced by quotient topology.
The corresponding dual space is then defined to be the space of currents on $X$.
\begin{mydef}(D\'efinition 1.9 \cite{Dem85})
A locally integrable function $V$ (with respect to area measure induced from Lebesgue measure of any local embedding) over a complex space $Z$ is called weakly psh (resp. weakly quasi-psh) if $V$ is locally bounded from above and $i \d \dbar V \geq 0$ in the sense of currents (resp. $i \d \dbar V \geq \alpha$ in the sense of currents with $\alpha$ a smooth form on $Z$).

\end{mydef}
When the complex space $Z$ is smooth, the condition that $V$ is locally bounded from above follows from the condition that $V$ is locally integrable and $i \d \dbar V \geq 0$ (resp. $i \d \dbar V \geq \alpha$ in the sense of currents with $\alpha$ a smooth form on $Z$).
However, this is not always the case in the singular setting.

We have also the definition of the psh function over a complex space.
\begin{mydef}
(D\'efinition 1.9 \cite{Dem85})
Let $V: Z \to [-\infty, \infty[$ be a function that is not identically infinite over any open set of $Z$.
Then $V$ is called psh (resp. quasi-psh) if for any local embedding $j: Z \hookrightarrow \Omega \subset \C^N$, $V$ is the local restriction of a psh (resp. quasi-psh) function on $\Omega$.
\end{mydef}
We have the following equivalent definition of psh functions due to J.E. Fornaess and R. Narasimhan.
\begin{mythm}
(Theorem 5.10 \cite{FN})

A function $V$ is a psh function over a complex space $X$ if and only if

(1) $V$ is upper semi-continuous.

(2) for any holomorphic map $f: \Delta \to X$ from the unit disc $\Delta$,
either $V \circ f$ is subharmonic or $V \circ f$ is identically infinite. 
\end{mythm}
As a direct consequence, the pullback of a quasi-psh function between complex spaces is still quasi-psh.

Of course, over a complex manifold, the definitions of the psh function and weakly psh function coincide.
However, the definitions of psh and weakly psh function are different in general over a complex space.
\begin{myex}
{\em
Let $Z \subset \C^2$ be a complex space defined by $\{ z_1z_2=0 \}$. 
Consider a function $V$ which is identically equal to 1 on $\{z_1=0\}$ and identically equal to 0 otherwise.
Then $V$ can not be the restriction of some psh function on any open neighbourhood of 0 in $\C^2$.
Otherwise, the restriction of such a function on $\{z_2=0\}$  should be identically equal to 0, which contradicts its value at the origin.
We claim, however, $V$ is a weakly psh function on $Z$.
In fact $i \d \dbar V=0$.
It is enough to consider the open neighbourhood of the origin.
Let $g$ be a test function of $Z$ near the origin.
$\langle i \d \dbar V, g \rangle=\langle V, i \d \dbar g \rangle=\int_{\{z_1=0\} } i \d \dbar g=0$ for any $g$.
}
\end{myex}
We also recall the Bott-Chern cohomology class which is more precise than the previously considered de Rham cohomology.
By Lemma 4.6.1 of \cite{BEG},  any pluriharmonic distribution on a normal complex space $X$ is locally the real part of a
holomorphic function, i.e. the kernel of the $i \d \dbar$ operator on the sheaf of distributions of bidegree $(0,0)$ coincides with the sheaf $\R \cO_X$ of real parts of holomorphic germs.
The Bott–Chern cohomology space is defined to be
$$H^{1,1}_{BC}(X, \C):=H^1(X, \R \cO_X).$$
If $X$ is a complex orbifold, the complex of orbifold smooth forms gives a soft resolution of $\C_X$.
The usual definition of Bott-Chern complex using orbifold smooth forms or orbifold currents gives an exact sequence
$$0 \to \R\cO_X \to A^{0,0}_{X,orb} \xrightarrow{i \d \dbar} A^{1,1}_{X,orb} \to A^{2,1}_{X,orb} \oplus A^{1,2}_{X,orb} \to \cdots$$
where $A^{i,j}_{X,orb}$ means the sheaf of orbifold smooth forms or currents of bidegree $(i,j)$.
In particular, an element of $H^{1,1}_{BC}(X, \C)$ can be represented by global orbifold $(1,1)-$forms.
If $X$ is furthermore smooth, we find the usual definition of Bott-Chern cohomology defined by the Bott-Chern complex.

Let $\omega$ be a K\"aher form on $X$.
Locally $\omega$ is the restriction of $i \d \dbar$ of a smooth function such that the difference on the interestion with other local charts is pluriharmonic.
In other words, $\omega$ defines an element in
$H^0(X, C^\infty_X/\R\cO_X)$.
The short exact sequence
$$0 \to \R \cO_X \to C^\infty_X \to C^\infty_X/ \R \cO_X \to 0$$
implies that $H^1(X, \R \cO_X) \simeq H^0(X, C^\infty_X/\R\cO_X)/i \d \dbar A^{0,0}_X(X)$.
In particular, the K\"ahler form $\omega$ defines an element in $H^{1,1}_{BC}(X,\C)$.

Since the sheaf $\R \cO_X$ is quasi-isomorphic to $C^\infty_X \to C^\infty_X/ \R \cO_X$, the complex morphism
$$\begin{tikzcd}
0 \arrow[r] \arrow[d] & 0 \arrow[d] \arrow[r] & A^0_X \arrow[d, "-\d"] \arrow[r] & A^0_X/\R\cO_X \arrow[r] \arrow[d, "\d \dbar"] & 0 \arrow[d] \arrow[r] & \cdots \\
0 \arrow[r]           & A^0_X \arrow[r]       & A^1_X \arrow[r, "d"]             & A^2_X \arrow[r]                               & A^3_X \arrow[r]       & \cdots
\end{tikzcd}$$
induces
$H^{1,1}_{BC}(X,\C) \to H^2_{\cD-\dR}(X,\C)$.
If we change the complex of smooth forms by the complex of smooth orbifold forms on the bottom line, we find the natural morphism $H^{1,1}_{BC}(X,\C) \to H^2(X,\C)$. 
Notice that the $\d \dbar$ lemma holds for compact K\"ahler orbifold.
The Bott-Chern cohomology is also isomorphic to the orbifold Dolbeault cohomology.
Since the Hodge decomposition holds for compact K\"ahler orbifold,
the Bott-Chern cohomology can be seen as subspace of $H^2(X,\C)$.
The natural morphism $H^{1,1}_{BC}(X,\C) \to H^2(X,\C)$ factorise through $H^2_{\cD-\dR}(X,\C)$
which implies that $H^{1,1}_{BC}(X,\C)$ can also be seen as a subspace of $H^2_{\cD-\dR}(X,\C)$.

We will need the following Richberg regularisation theorem on complex space due to Varouchas \cite{Var}.
\begin{mythm}
Let $X$  be a complex space.  Suppose it admits an open covering 
$(U_\alpha)_\alpha$ and a system of continuous strongly psh functions $\varphi_\alpha$ on $U_\alpha$ such that 
$\varphi_\alpha-\varphi_\beta$ is pluriharmonic on $U_\alpha \cap U_\beta$.
Here a strongly psh function means local restriction of a strongly psh function of the local ambient space.
Then there are smooth strongly psh functions $\tilde{\varphi}_\alpha$ on $U_\alpha$ such that 
$\tilde{\varphi}_\alpha-\tilde{\varphi}_\beta=\varphi_\alpha-\varphi_\beta$. 
In particular, the Bott-Chern cohomology class defined by $i \d \dbar \tilde{\varphi}_\alpha$ is the same as the class defined by $ i \d \dbar \varphi_\alpha$.  
\end{mythm}
\begin{myrem}
{\em 
Any global section of $C^0_X/\R \cO_X$ also defines an element of $H^{1,1}_{BC}(X,\C)$.
In fact, we have an isomorphism 
$$H^0(X, C^\infty/\R \cO_X)/ i \d \dbar C^\infty(X) \simeq H^0(X, C^0/\R \cO_X)/ i \d \dbar C^0(X) \simeq H^{1,1}_{BC}(X,\C).$$
In literature, any global section $H^0(X, \cD'_X/\R \cO_X)$ is called a $(1,1)-$current with local potential.
Similarly,  any global section $H^0(X, C^0_X/\R \cO_X)$ can be called a $(1,1)-$current with continous local potential.
It is K\"ahler if the local potentials are strongly psh.
The above theorem of Varouchas can be reformulated as follows:

{\it
The open cone generated by K\"ahler forms in the space $H^0(X, C^\infty_X/\R \cO_X)/ i \d \dbar C^\infty(X)$ coincide with the open cone generated by K\"ahler currents with continuous local potentials in the space $H^0(X, C^0_X/\R \cO_X)/ i \d \dbar C^0(X)$.
}

Notice that the Bott-Chern cohomology is also the hypercohomology of some complex in terms of orbifold currents instead of orbifold smooth forms.
Thus we can ask similarly, whether the open cone generated by K\"ahler orbifold forms coincides with the one generated by K\"ahler orbifold currents with continuous potentials.
It is true by the following arguments.
Let $\alpha$ be an orbifold smooth form on a complex orbifold $X$.
Let $\varphi$ be a continuous orbifold function such that 
$$\alpha+i \d \dbar \varphi \geq 0$$
in the sense of currents.
Let $P$ be the total space of the frame bundle with natural projection $\pi: P \to X$.
The construction can be found for example in Section 2.4 of \cite{MM10}.
Then $\pi$ is a $\GL(n,\C)-$equivariant morphism with group action on $P$ induced from the right multiplication of $\GL(n,\C)$.
Locally let $U/G$ be an open set of $X$ with $U \subset \C^n$ smooth and $G$ a finite subgroup of $\SO(2n)$.
The preimage $\pi^{-1}(U/G)$ is isomorphic holomorphically to the orbifold $U \times \GL(n,\C)/G$ where the action of $g \in G$ on $(x, A) \in U \times \GL(n,\C)$ is given by $g \cdot (x,A)=(g \cdot x, gA)$.
Since the action of $G$ on $U \times \GL(n,\C)$ is free, $P$ is a smooth manifold.
Fix a smooth Hermitian form $\omega'$ on $P$.

Thus $\varphi \circ \pi$ is a continous $\pi^* \alpha-$psh function on $P$.
By classical Richberg regularisation theorem \cite{Ric68} \cite{GW75}, there exists $\psi_\varepsilon \in C^\infty(P)$ such that
$|\varphi \circ \pi-\psi_\varepsilon| \leq \varepsilon$ and
$$\pi^* \alpha+i \d \dbar \psi_\varepsilon \geq -\varepsilon \omega'.$$
Let $dg$ be the Haar measure of $\GL(n,\C)$ with total mass 1.
Without loss of generality, we can assume that $\omega'$ is integrable with respect to the Haar measure.
Define $\varphi_\varepsilon:=\int \psi_\varepsilon dg \in C^\infty(P)$.
Since it is $\GL(n,\C)-$invariant, it descends to a smooth orbifold function on $X$.
Since $\psi_\varepsilon$ is bounded on each orbit of $\GL(n,\C)$,
the integration is finite.
We have that
$$\alpha+i \d \dbar \varphi_\varepsilon \geq - \varepsilon \int  \omega' dg$$
as orbifold smooth forms on $X$. 
}
\end{myrem}
Now we prove the inverse of Lemma 1 by the results of Varouchas.
\begin{myprop}
Let $X$ be a complex orbifold and $\omega$ a K\"ahler orbifold form on $X$.
Assume $X$ is paracompact and second countable.
Then $X$ is a K\"ahler complex space.
\end{myprop}
\begin{proof}
 Let $(U_\alpha)_{\alpha \in \N}$ be a locally finite cover of $X$ such that $U_\alpha \simeq \tilde{U}_\alpha/G_\alpha$ where $\tilde{U}_\alpha$ are simply connected euclidean open sets and $G_\alpha \subset GL_n(\C)$ are small.
 By assumption, over $U_\alpha$, $\omega$ is given $i \d \dbar \tilde{\varphi}_\alpha$ for some $\tilde{\varphi}_\alpha \in C^\infty(\tilde{U}_\alpha)$ which is $G_\alpha-$invariant.
 Since $U_\alpha$ as the topology space is given by the quotient topology of $\tilde{U}_\alpha$,
 $\tilde{\varphi}_\alpha$ defines some continous function $\varphi_\alpha$ on $U_\alpha$.
 
 We claim that $\varphi_\alpha$ is strictly psh on $U_\alpha$.
 Let $\psi_\alpha$ be a continuous strictly psh function on $U_\alpha$.
 (For example, we can embed $U_\alpha$ in some euclidean space and take $\psi_\alpha$ to be the restriction of some continuous strictly psh function on the euclidean space.)
 Without loss of generality, we can assume that $\tilde{\varphi}_\alpha, \pi_\alpha^* \psi_\alpha$ are defined on the closure of $\tilde{U}_\alpha$.
 Then $\tilde{\varphi}_\alpha-\varepsilon \pi_\alpha^* \psi_\alpha$ is psh on $\tilde{U}_\alpha$ for $\varepsilon>0$ small enough
 since $\tilde{\varphi}_\alpha$ is strongly psh.
Since $\pi_\alpha$ is an unramified cover over the regular part, $\varphi_\alpha-\varepsilon \psi_\alpha$ is psh on the regular part of $U_\alpha$.
Since $\varphi_\alpha-\varepsilon \psi_\alpha$ is continuous, it is psh on $U_\alpha$ by Theorem 3.
 Thus $\varphi_\alpha=(\varphi_\alpha-\varepsilon \psi_\alpha)+\varepsilon \psi_\alpha$ is strongly psh.
 
 By the diagram in Proposition 1, it is easy to see that $\varphi_\alpha-\varphi_\beta$ is pluriharmonic on $U_\alpha \cap U_\beta$.
 Thus by the result of Varouchas (Theorem 4), $X$ is a K\"ahler complex space.
\end{proof}
By combining Lemma 1 and Proposition 4, we see that if $X$ is a compact orbifold, it is K\"ahler as a complex space if and only if it is K\"ahler as an orbifold.
\begin{mycor}
Let $X$ be a smooth compact manifold with $G$ a finite subgroup of biholomorphisms.
Then $X/G$ is a K\"ahler orbifold if and only if $X$ is K\"ahler.
In particular, $X/G$ is projective if and only if $X$ is projective.
\end{mycor}
\begin{proof}
Note that locally $X/G$ is quotient of some Euclidean coordinate chart of $X$ under some linear group action of $G$ (thus $X/G$ is a complex orbifold) by \cite{Car57}.

 The ``if'' part is trivial. 
 Notice that in this case, we can apply directly Corollary 3.2.1 of \cite{Var} to conclude that $X/G$ is a K\"ahler complex space. 
 Notice that as pointed out in Remark 8.5 of \cite{DHP}, the image of a K\"ahler manifold under a proper surjective morphism is not necessarily K\"ahler.
 Here the result of Varouchas applies since the quotient map $X \to X/G$ is geometrically flat.
 
 Conversely, without loss of generality, we can assume that the action is effective in codimension 1 which means that $\cup_{g \in G, g \neq e} \{x \in X, g \cdot x=x\}$ is codimension at least 2 in $X$.
 Otherwise, the subgroup $H$ generated by non effective in codimension 1 elements forms a normal subgroup of $G$.
 Thus $X /G \simeq (X/H)/(G/H)$ with $X/H$ smooth.
 By the result of \cite{Bin83}, the (flat finite) quotient map $X \to X/H$ is a K\"ahler morphism.
In particular, $X$ is K\"ahler if and only if $X/H$ is K\"ahler.
 If $X/G$ is K\"ahler, by Lemma 4.2.2 of \cite{Var}, $X$ is weakly
K\"ahler. Since $X$ is smooth, $X$ is in fact K\"ahler.

 
 For the last claim, note that an ample line bundle $A$ induces an ample orbifold line bundle $\otimes_{g \in G}g^* A$ on $X/G$. Thus $X/G$ is projective.
 Conversely, the embedding $X/G$ into some projective space induces an orbifold K\"ahler form on $X/G$ (as constructed in Lemma 1) such that the pullback of this K\"ahler form is valued in $H^{1,1}(X) \cap H^2(X,\Z)$.
\end{proof}
Notice that the corollary gives an effective way to construct singular K\"ahler complex space.
The following examples show that in general, a projective variety with quotient singularities is not necessarily a global quotient.
\begin{myex}
{\em
We first recall some basic results about weighted projective space.
For more information, we refer to \cite{Dol82}.
A subvariety $X$ of some weighted projective space $\P$ of dimension $N$ is called quasi smooth if $\pi^{-1}(X)$ is smooth in $\C^{N+1} \setminus \{0\}$ with natural projection $\pi:\C^{N+1} \setminus \{0\} \to \P$ which we always assume in the following.
Let $C(X)$ be the closure of $\pi^{-1}(X)$ in $\C^{N+1}$.
Notice that $X$ is not necessarily smooth if it is quasi-smooth since $X \simeq C(X) \setminus \{0\}/\C^*$ where the $\C^*-$action may have fixed points.
But by Theorem 3.1.6 of \cite{Dol82}, $X$ is a complex orbifold.
By Lemma 3.2.2 \cite{Dol82}, if the dimension of $X$ is bigger than two, $C(X) \setminus \{0\}$ is simply connected.
Since the singular part of $X$ is of codimension at least 2 which implies that $\pi^{-1}(X_{\mathrm{sing}})$ is of codimension at leat 2 in smooth manifold $C(X) \setminus \{0\}$, $\pi^{-1}(X_{\mathrm{reg}})$ is also simply connected.
By homotopy exact sequence of the fibration $\pi: \pi^{-1}(X_{\mathrm{reg}}) \to X_{\mathrm{reg}}$,
$X_{\mathrm{reg}}$ is simply connected.

We claim that if $X$ is not smooth, $X$ is not isomorphic to $Y/H$ for some $Y$ smooth manifold with holomorphic non effective in codimension 1 action of some finite group $H$.
Otherwise, denote $p: Y \to X$ the induced quotient map.
The quotient map $p$ is an unramified cover over $X_{\mathrm{reg}}$ of order equal to the caradinal of $H$.
Since $X_{\mathrm{reg}}$ is simply connected, $H$ has to be trivial.
It implies that $X$ is smooth. Contradiction.
}
\end{myex}
One can construct more K\"ahler complex space by the work of \cite{GH10}.

\begin{myprop}
Let $\cF$ be a reflexive sheaf over a compact orbifold $X$.
Then $\cF$ is stable (resp. semistable) with respect to all orbifold subsheaves if and only if $\cF$ is stable (resp. semistable).
\end{myprop}
\begin{proof}
Let $\cG$ be a subsheaf of $\cF$. There exists a natural morphism $(\pi_\alpha^* \cG/\tors)_\alpha \to (\pi_\alpha^* \cF/\tors)_\alpha$ of orbifold sheaves.
However, it is not necessarily injective.
The image is a subsheaf of $(\pi_\alpha^* \cF/\tors)_\alpha$ which is isomorphic to $(\pi_\alpha^* \cG/\tors)_\alpha$ on the preimages of regular part.
In particular, the image has the same determinant orbifold line bundle as $(\pi_\alpha^* \cG/\tors)_\alpha$.
Hence they have the same first Chern class.
If $\cF$ is stable (resp. semistable) with respect to all orbifold subsheaves, it is stable (resp. semistable). 

Conversely, let $(\cG_\alpha)_\alpha$ be an orbifold subsheaf of  $(\pi_\alpha^* \cF/\tors)_\alpha$. 
$(\pi_{\alpha,*} \cG_\alpha)^{G_\alpha}$ is a subsheaf of $ \cF \simeq (\pi_{\alpha,*} \pi^*_\alpha \cF/\tors)^{G_\alpha}$.
The natural morphism
$$\pi_\alpha^* ((\pi_{\alpha,*} \cG_\alpha)^{G_\alpha})/\tors \to \pi_\alpha^* \pi_{\alpha,*} \cG_\alpha /\tors \to \cG_\alpha$$
is isomorphic in codimension 1 and both sides have the same orbifold determinant line bundle.
If $\cF$ is stable (resp. semistable), $\cF$ is stable (resp. semistable) with respect to all orbifold subsheaves.
\end{proof}
Now, notice that all tools of PDE theory work for a compact K\"ahler orbifold with suitable modifications (e.g. Sobolev inequality, heat kernel estimate etc.).
In particular, we have the following result.
\begin{myprop}
Let $E$ be an orbifold vector bundle over a compact K\"ahler complex space $(X, \omega)$ with quotient singularity.
Assume that $E$ viewing as an orbifold sheaf is $\omega-$polystable for all orbifold subsheaves.
Then the Bogomolov inequality for $E$ holds.
\end{myprop}
\begin{proof}
By the proof of Lemma 1, there exists a sequence of orbifold K\"ahler forms $\omega_\varepsilon$ converging to $\omega$ in the topology of ``orbifold" smooth forms and in the same singular cohomology class.
In particular, since $E$ is $\omega-$polystable, 
$E$ is also $\omega_\varepsilon-$polystable.
By Theorem 4.1 of \cite{Fau19},
for any $\varepsilon>0$ small enough such that $\omega_\varepsilon$ is an ``orbifold" K\"ahler form,
the orbifold version of Uhlenbeck-Yau theorem shows the existence of $\omega_\varepsilon-$Hermitian-Einstein metric.
In particular,
 the Bogomolov inequality for $E$ with respect to $\omega_\varepsilon$ holds.
Since $\omega_\varepsilon$ are in the same singular cohomology class of $\omega$,
 the Bogomolov inequality for $E$ with respect to $\omega$ also holds.
\end{proof}
\section{Regularisation on compact complex orbifold}
In this section, we give the variants of Demailly's regularisation on compact complex orbifolds.

We start with the regularisation of almost positive orbifold current by orbifold smooth forms.
Let $X$ be a compact complex orbifold.
We can define the fibre-holomorphic part of the exponential map
$$\exph: T_X \to X$$
as in \cite{Dem21}.
We briefly recall the construction.
Let $h$ be a (real analytic) Hermitian orbifold metric on $X$.
(The existence can be shown for example by viewing $X$ as a real analytic orbifold by \cite{Kan13}.)
Consider the exponential map associated with the Chern connection of the metric $h$ (cf. \cite{CR02} for definition of exponential map of a compact orbifold).
It is an orbifold morphism between $T_X$ and $X$ which is real analytic.
We define the fibre-holomorphic part of the exponential map $\exph$ which is uniquely defined on a tubular neighbourhood of the zero section of $T_X$.

Now as in \cite{Dem94}, we have the following regularisation theorem in the context of orbifold.
\begin{mythm}
Let $T$ be a closed almost positive $(1,1)-$orbifold current over a compact complex orbifold $X$ and
let $\alpha$ be a smooth real $(1,1)-$orbifold form in the same $\d \dbar-$cohomology class as $T$, i.e.
$T = \alpha + i \d \dbar \Psi$ where $\Psi$ is an almost psh orbifold function (i.e. almost psh on the local ramified smooth cover). Let $\gamma$ be a continuous real
$(1,1)-$orbifold form such that $T\geq \gamma$. Suppose that $T_X$ is equipped with a smooth Hermitian orbifold
metric $\omega$ such that the orbifold Chern curvature form $\Theta(T_X)$ satisfies
$$(\Theta(T_X ) + u \otimes \Id_{T_X}
)(\theta \otimes \xi,\theta \otimes \xi)  \geq 0, \forall \theta, \xi \in T_X \; \mathrm{with} \;\langle \theta, \xi \rangle= 0,$$
for some continuous nonnegative $(1,1)-$orbifold form $u$ on $X$. Then there is a family of
closed almost positive (1,1)-orbifold forms $T_\varepsilon=\alpha+i \d \dbar \psi_\varepsilon$, such that $\psi_\varepsilon, \forall \varepsilon \in ]0, \varepsilon_0[$ is orbifold
smooth over $X$, increases with $\varepsilon$, and converges to $\Psi$ as $\varepsilon$ tends to 0 (in particular,
$T_\varepsilon$ is orbifold smooth and converges weakly to $T$ on $X$), and such that

(i) $T_\varepsilon \geq \gamma -\lambda_\varepsilon u-\delta_\varepsilon \omega$ where:

(ii) $\lambda_\varepsilon (x)$ is an increasing family of continuous functions on $X$ such that
$\lim_{\varepsilon \to 0} \lambda_\varepsilon (x) = \nu(T,x)$ (Lelong number of $T$ at $x$, defined on the local ramified smooth cover) at every point,

(iii) $\delta_\varepsilon$ is an increasing family of positive constants such that $\lim_{\varepsilon \to 0} \delta_\varepsilon = 0$.

On the other hand, we have the following singularity attenuation process.
For every $c > 0$, there
is a family of closed almost positive (1,1)-orbifold currents $T_{c,\varepsilon }= \alpha + i\d \dbar \psi_{c,\varepsilon}, \varepsilon \in ]0,\varepsilon_0[$,
such that $\psi_{c,\varepsilon}$ is orbifold smooth on $X \setminus E_c(T)$, increasing with respect to $\varepsilon$, and converges
to $\Psi$ as $\varepsilon$ tends to 0 (in particular, the current $T_{c,\varepsilon}$ is orbifold smooth on $X \setminus E_c(T)$ and
converges weakly to $T$ on $X$), and such that

(i) $T_{c,\varepsilon}\geq \gamma-\min(\lambda_\varepsilon,c)u -\delta_\varepsilon \omega$ where:

(ii) $\lambda_\varepsilon(x)$ is an increasing family of continuous functions on $X$ such that
$\lim_{\varepsilon \to 0} \lambda_\varepsilon (x) = \nu(T,x)$

(iii) $\delta_\varepsilon$ is an increasing family of positive constants such that $\lim_{\varepsilon \to 0} \delta_\varepsilon = 0$,

(iv) $\nu(T_{c,\varepsilon},x) = (\nu(T,x) -c)_+$ at every point $x \in X$.
\end{mythm}
\begin{proof}
Select a cut-off function $\chi: \R \to \R$ of
class $C^\infty$ such that
$\chi(t) > 0$ for $t < 1$, $\chi(t) = 0$ for $t \geq 1$,
$\int_{v \in \C^n}
\chi(|v|^2) d \lambda(v) = 1$.

Define the regularisation of $\Psi$ by
$$\psi_\varepsilon(z):=\frac{1}{\varepsilon^{2n}} \int_{\xi \in T_{X,z}} \Psi(\exph(z,\xi)) \chi(\frac{|\xi|^2}{\varepsilon^2}) d \lambda(\xi)$$
which is well defined for $\varepsilon>0$ small enough.

The same careful calculation as in the proof of Theorem 4.1 of \cite{Dem94} shows the first result.
The singularity attenuation process is the same as Theorem 6.1 of \cite{Dem94} following the ideas of \cite{Kis79}.
\end{proof}
We have also the regularisation by currents with analytic singularities as shown as follows.
\begin{mythm}
Let $\varphi$ be an almost psh orbifold function on a compact complex orbifold
$X$ such that 
$i \d \dbar \varphi \geq \gamma$ for some continuous (1,1)-orbifold form $\gamma$. Then there is a sequence of
almost psh orbifold functions $\varphi_m$ such that $\varphi_m$ has the same singularities (on the local ramified smooth cover) as a logarithm of a sum
of squares of holomorphic functions and a decreasing sequence $\varepsilon_m > 0$ converging to 0
such that

(i) $\varphi(x) \leq \varphi_m(x) \leq \sup_{|\xi-x|<r}
\varphi(\xi) + C
(|\log r|/
m + r + \varepsilon_m
)$
with respect to ramified smooth coordinate open sets covering $X$. In particular, $\varphi_m$ converges to $\varphi$
pointwise and in $L^1(X)$ and

(ii) $\nu(\varphi,x) -\frac{n}{m} \leq \nu(\varphi_m,x) \leq \nu(\varphi,x)$ for every $x \in X$;

(iii)$i \d \dbar \varphi_m \geq \gamma -\varepsilon_m \omega$ for some fixed smooth Hermitian orbifold metric $\omega$.
\end{mythm}
\begin{proof}
The construction is similar to the case of the manifold as shown in Proposition 3.9 of \cite{Dem92}.
Since the construction is not canonical, some modifications are needed to ensure the approximation function is invariant under the local action of the finite group.
We give the outline of the construction and the verification of the properties is almost identical to the case of manifolds which shall be omitted.

We select a finite covering $(W_\nu/G_\nu)$ of $X$ with open coordinate charts $W_\nu$ and finite groups $G_\nu \subset U(n)$. Given $\delta > 0$, we
take in each $W_\nu$ a maximal family of points with (coordinate) distance to the boundary $\geq 3 \delta$ and mutual distance $\geq \delta/2$. 
In this way, we get for $\delta > 0$ small a finite covering
of $X$ by the images of open balls $U'_j$ of radius $\delta$ under the quotient maps, such that
the concentric ball $U_j$ of radius $2 \delta$ is relatively compact in the corresponding chart $W_\nu$.
Notice that $U'_j$ or $U_j$ is not necessarily invariant under the local action.
Since the topology of $X$ is locally given by quotient topology, the image of the small balls is open in $X$.
Let $\tau_j : U_j\to B(a_j,2\delta)$ be the isomorphism given by the coordinates of $W_\nu$. 

Let
$\varepsilon(\delta)$ be a modulus of continuity for the orbifold form $\gamma$ on the sets $U_j$, such that $\lim_{\delta \to 0} \varepsilon(\delta) = 0$ and
$\gamma_x -\gamma_{x'} \leq 2 \varepsilon(\delta) \omega_x$
 for all $x,x' \in U_j$. We denote by $\gamma_j$ the (1,1)-orbifold form with constant
coefficients on $B(a_j,2\delta)$ such that $\tau^*_j \gamma_j$ coincides with $\gamma -\varepsilon(\delta) \omega$ at $\tau^{-1}_
j (a_j)$. 

We set $\varphi_j:=\varphi \circ \tau_j^{-1}$ on $B(a_j,2\delta)$ and let $q_j$ be the unique homogeneous
quadratic function in $z -a_j$ such that $i \d \dbar q_j=\gamma_j$ on $B(a_j,2\delta)$. Finally, we set
$ \psi_j(z) = \varphi_j(z) -q_j(z)$ on $B(a_j,2\delta)$.
Take Bergman approximation $\psi_{j,m}$ of the psh function $\psi_j$.
Notice that we have for any $g \in G_\nu$ (such that $U_j \subset W_\nu$),
$g^*(\psi_{j,m} \circ \tau_j)$ is the Bergman approximation of $g^*(\psi_{j} \circ \tau_j)$.
In particular, $w_{j,m}:=(\psi_{j,m}+q_j) \circ \tau_j$ is invariant under any subgroup $G'_{\nu}$ of $G_\nu$ over $\cap_{g \in G'_\nu} g U_j$.
Thus $w_{j,m}$ descends to a function on the quotient of $U_j$ (i.e. $(\cup_{g \in G_\nu} g U_j)/G_\nu$).

We let $U'_j \subset \subset U''_j \subset \subset U_j$ be concentric balls of radii $\delta, 1.5 \delta, 2\delta$ respectively. 
By almost the same proof of the manifold case, we have that there exist constants $C_{j,k}$ independent of $m$ and $\delta$ such that the almost
psh functions satisfy
$$|w_{j,m}-w_{k,m}| \leq C_{j,k}
(\log \delta^{-1} + m \varepsilon(\delta)\delta^2)$$ on the intersection of the quotient of $U''_j$  with the quotient of $U''_k $.

To glue a global almost psh orbifold function, take $\theta_j$ smooth nonnegative functions with support
in $\cup_{g \in G_\nu} g U''_j$, such that $\theta_j \leq 1$ on $\cup_{g \in G_\nu} g U''_j$ and $\theta_j = 1$ on $U'_j$ which is is invariant under $G_\nu-$action.
Define
$$\tilde{w}_{j,m}(x) = w_{j,m}(x) + 2m
(\frac{C_1}{m} + C_3 \varepsilon(\delta)\sup_{g \in G_\nu, x \in g^{-1}U_j} (\delta^2/2 -|
(\tau_j(g \cdot x)-a_j)|^2)$$
for $C_3$ sufficiently large and 
$C_1>0$ such that $\psi_{j,m} \geq \psi_j -\frac{C_1}{m}$.
Notice that $(\delta^2/2 -|\tau_j(x)-a_j|^2)$ attains strictly minumum on the boundary of $U_j$.
Thus locally the supremum is always taken as the supremum of finite almost psh functions which is thus almost psh.
Since $\tilde{w}_{j,m}$ is defined on $U_j$ and invariant,
it can also be viewed as a function over  $(\cup_{g \in G_\nu} g U_j)/G_\nu$.
Notice that for $x \in \cup_{g \in G_\nu} g U''_j \setminus \cup_{g \in G_\nu} g U'_j$,
$$\tilde{w}_{j,m}(x) \leq \sup_{k \neq j, x \in U'_k} \tilde{w}_{k,m}(x)$$
following the same proof of inequality without taking superimum for any $g \in G_\nu$ as in \cite{Dem92}.
Notice also that the maximum number $N$ of overlapping balls of forms $g U_j$ is bounded by the maximum number of overlapping balls $ U_j$ times the maximal cardinal of $G_{\nu}$.

It is easy to check that Lemma 3.5 of \cite{Dem92} applies in the orbifold setting.
Define
$$\varphi_m :=\frac{1}{2m} \log(\sum_j \theta_j^2 e^{\tilde{w}_{j,m}})$$
which is a almost psh orbifold function and gives the approximation.
\end{proof}
Moreover, we have the following equisingular approximation as in \cite{DPS01}.
The proof is analogous to that of Theorem 6.
\begin{mythm}
Let $\varphi$ be an almost psh orbifold function on a compact complex orbifold
$X$ such that 
$i \d \dbar \varphi \geq \gamma$ for some continuous (1,1)-orbifold form $\gamma$. Then there is a sequence of
almost psh orbifold functions $\varphi_m$ such that $\varphi=\lim_{m \to \infty}\varphi_m$
and

(i) $\varphi_m$ is orbifold smooth in the complement $X \setminus Z_m$
of an analytic set $Z_m \subset X$;

(ii) $\varphi_m$ is a decreasing sequence, and $Z
_m \subset Z_{m+1}$ for all $m$;

(iii)
$\int_X (e^{-2\varphi} - e^{-2\varphi_m} ) dV_\omega$ is finite for every $m$ and converges to 0 as $m \to \infty$;

(iv)
$\cI(\varphi_m)=\cI(\varphi)$
for all $m$ (“equisingularity”) on the local ramified smooth cover;

(v) $i \d \dbar \varphi_m \geq \gamma-\epsilon_m \omega$
where $\lim_{m \to \infty} \epsilon_m=0$.
\end{mythm}
As consequence, we have also the hard Lefschetz theorem for pseudoeffective orbifold line bundle over compact K\"ahler orbifold as in \cite{DPS01}.
\section{Strongly pseudoeffective orbifold vector bundle}
In this section, we generalise the results of \cite{Wu20} to the case of compact K\"ahler orbifold.

In the definition of the pseudoeffective vector bundle on smooth manifold, an additional condition is made on the approximate singular metrics on the tautological line bundle.
One may ask whether the additional condition could be made directly on some positive singular metric on the tautological line bundle.
For example, let $h$ be a positive singular metric on $\cO_{\P(E)}(1)$, could we consider the condition that the projection of singular set $\{h = \infty\}$ in $X$ is a (complete) pluripolar set?
However, this kind of condition does not behave well functorially as shown below.
In particular, the image of a (complete) pluripolar set under a proper morphism is not necessarily (complete) pluripolar as shown in the following Example 8.
Notice that by Remmert's proper mapping theorem, the image of a closed analytic set under a proper morphism is always (closed) analytic.
The projection of singular set $\{h = \infty\}$ in $X$ could be very bad which forbids the Skoda-El Mir type of extension theorem.

Recall that a subset $F$ of a complex manifold $X$ is said to be pluripolar if for
each point $z \in F$ there is a plurisubharmonic function $u$, $u \neq - \infty$, defined in an open
neighborhood $U$ of $z$, such that $F \cap U \subset \{u = -\infty\}$. 
A subset $F$ of a complex manifold $X$ will be called to be a complete pluripolar set if there is a
non-constant quasi-plurisubharmonic function $u$ defined on $X$ such that $F = \{u =-\infty \}$.
By the results of \cite{Col90}, this definition is equivalent to the condition that for
each point $z \in F$, there is a plurisubharmonic function $u$, $u \neq - \infty$, defined in an open
neighbourhood $U$ of $z$, such that $F \cap U = \{u = -\infty\}$ (i.e. complete locally pluripolar) if $F$ is closed.
In fact, they prove the following lemma.
\begin{mylem}(Lemma 1, \cite{Col90})
Let $X$ be a complex space and $A \subset X$ a closed complete locally pluripolar set.
Then $A$ can be defined by $\{U_i, \varphi_i\}_{i \in \N}$ with $U_i$ open and relatively compact in $X$, $\{U_i\}_{i \in \N}$
locally finite.
There exist $U''_i \Subset U'_i \Subset U_i \Subset X$ such that
$X = \cup_i U''_i$, $\varphi_i: U_i \to [-\infty, \infty[$ plurisubharmonic functions,
$$A \cap U_i=\{\varphi_i=-\infty\},$$
$\exp(\varphi_i)$  continuous, $\varphi_i$ smooth outside $A$ and $\varphi_i-\varphi_j$ bounded on $U'_i \cap U'_j \setminus A$.
\end{mylem}
Since $\varphi_i-\varphi_j$ is bounded on $U'_i \cap U'_j \setminus A$,
we can choose $p_i \in C_c^\infty(U'_i)$ such that 
$$\varphi_i+p_i < \varphi_j+p_j$$
on $\d U'_i \cap U''_j \setminus A$ and $p_i \geq 0$.
In particular, 
$$\psi(z):=\max\{\varphi_i(z)+p_i(z); z \in U'_i\}$$
is a quasi psh function as the maximum of finite quasi psh functions locally.
Note that $A$ is the pole set of $\psi$.

 Notice that the image of a pluripolar set under a proper morphism is not necessarily pluripolar.
 For example, let $Z$ be a non-pluripolar set in $\P^1$.
 Consider $A:=Z \times \{0\} \subset \P^1 \times \P^1$.
 Since $A \subset \P^1 \times \{0\}$, $A$ is pluripolar.
 However, the image of $A$ under the projection of $\P^1 \times \P^1$ onto its first component is not pluripolar.
 The statement is still false if we change the condition ``pluripolar" to ``complete pluripolar"
 as shown in the following example recommended to the author by Demailly.
\begin{myex}
{\em 
In this example, we construct a complete pluripolar subset $A \subset \P^1 \times \P^1$ such that the image of $A$ under the projection onto the first component is $\Q \subset \C \subset \P^1$.
Notice that as countable union of pluripolar sets in $\C$, $\Q$ is pluripolar.
On the other side, $\Q$ is not complete pluripolar. Otherwises, there exists a K\"ahler form $\omega$ on $\P^1$ and a quasi-psh function $\varphi$ on $\P^1$ such that $\Q=\{\varphi=-\infty\}$ and $\omega + i \d \dbar \varphi \geq 0$ in the sense of currents.
In particular, over $\C$, $\omega|_\C=i \d \dbar \psi$ for some $\psi \in C^\infty(\C)$ and $\varphi+\psi$ is psh on $\C$ with pole set $\Q$. 
Thus, $\Q$ should be a $G_\delta$ subset in $\C$ (i.e. $\Q=\cap_{j \in \N} V_j$ for $V_j$ open in $\C$).
Notice that $\Q$ is dense in $\C$ as well as all the $V_j$'s.
Thus we have that 
$$\emptyset= \Q \cap (\C \setminus \Q)=\cap_{j \in \N} V_j \cap \cap_{q\in \Q} \C \setminus \{q\} $$
which contradicts the Baire category theorem.

Define 
$$A := \cup_{a \in \N^*, b\in \Z} \{z_{a,b}= (b/a,a)\} \subset \C^2 \subset \P^1 \times \P^1.$$
Define for any $z \in \C^2$,
$$\delta(z):= \inf_{z \neq z_{a,b}} \log|z-z_{a,b}|^2$$
which is not infinite on $\C^2 \setminus A$ implying that $A$ is discrete in $\C^2$.

Consider a sequence of $a_{a,b}>0$ such that $\sum_{a,b} a_{a,b} < \infty$
which would be chosen later.
Consider 
$$\varphi'(z):=\sum_{a,b} a_{a,b} (\log|z-z_{a,b}|^2-\log(1+|z|^2)).$$
Notice that for any $z \notin A$, $\varphi'(z) \geq \sum_{a,b} a_{a,b} (\delta(z)-\log(1+|z|^2)) \neq - \infty$
 and for any $z_{a,b}$, $\varphi'(z_{a,b})=-\infty$.
 In particular, $A \cap\C^2 =\{ \varphi'=-\infty\}$.
Consider
$$C_{a,b}=\sup_{z \in\C^2}( \log|z-z_{a,b}|^2-\log(1+|z|^2) )< \infty.$$
With suitable choice of $a_{a,b}>0$, we can assume that $\sum_{a,b} a_{a,b} C_{a,b} < \infty$.
In particular, $\varphi'$ is uniformly bounded from above on $\C^2$ which extends to be a quasi-psh function on $\P^1 \times \P^1$ with pole set $A$.
 }
\end{myex}
Now let us generalise the structure theorem of numerically flat vector bundle (Theorem 1.18 \cite{DPS94}) to compact K\"ahler orbifolds.
The proof is almost identical to the original proof other than the estimation of the second Chern class using Corollary 2.6 \cite{DPS94}.
Notice that the proof of Corollary 2.6 \cite{DPS94} uses a resolution of singularities to reduce to the smooth case which is highly non-trivial in the orbifold case.
In other words, could we always have a resolution of singularities of an orbifold by orbifold morphisms (such that the pullback of the orbifold vector bundle is always well-defined)? 
We will need the following lemma.
\begin{mylem}
Let $E$ be a nef orbifold vector bundle over a compact complex orbifold.
Let $\cQ$ be a quotient orbifold sheaf of $E$ of rank $r$.
Then $c_1(\cQ)$ is pseudoeffective (i.e. it contains a positive orbifold $(1,1)-$current).
\end{mylem}
\begin{proof}
Notice that the determinant orbifold line bundle of a torsion orbifold sheaf admits a global non-trivial section which implies in particular that its first Chern class is pseudoeffective.
Thus without loss of generality, we can assume that $\cQ$ is torsion free in which case $\det(\cQ)=(\wedge^r \cQ)^{**}$ over the local ramified smooth cover.

The surjection $\wedge^r E \to \wedge^r \cQ$ induces an injection $\det(\cQ)^* \to \wedge^r E^*$.
In particular, the smooth metrics $h_\varepsilon$ on $E$ induces a possible singular metric on $\det(\cQ)$ whose curvature is locally given by $T_\varepsilon:=i \d \dbar \log |s|^2_{\wedge^r h_\varepsilon^*}$
where $s$ is a local generator of $\det(\cQ)^*$.
Since $E$ is nef, $T_\varepsilon \geq -\varepsilon \omega$ in the sense of currents with some reference Hermitian metric $\omega$.
The weak compactness implies that $c_1({\det(\cQ)})$ is pseudoeffective.
\end{proof}
We still have the following calculation over compact complex orbifolds, which can be seen as a direct consequence of intersection theory, and is still valid on the level of forms without passing to cohomology classes: for every $k \in \N$
$$\pi_* \left(\frac{i}{2 \pi} \Theta (\cO_{\P(E)}(1),h)\right)^{r+k}= s_k(E,h)$$
for any smooth orbifold metric $h$ on $E$ with $\pi: \P(E) \to X$.
Note that the Segre classes can be written in terms of Chern classes and the Chern classes can be represented by the Chern forms derived from the curvature tensor. 
For the manifold case, we refer for example to the papers \cite{Div16}, \cite{Gul12} and \cite{Mou04}.
As consequence, if $E$ is a nef orbifold vector bundle over $(X, \omega)$ a compact K\"ahler orbifold,
$$\pi_*\left(\frac{i}{2 \pi} \Theta (\cO_{\P(E)}(1),h_\varepsilon)+\varepsilon \pi^* \omega \right)^{r+2}= s_2(E,h_\varepsilon)+c_1(E,h_\varepsilon) \wedge (r+2) \varepsilon \omega+\frac{(r+2)(r+1)}{2} \varepsilon^2 \omega^2$$
are positive orbifold smooth forms which imply the existence of a positive orbifold current in $s_2(E)$.
In particular, we have Chern number estimate
$$c_2(E) \cdot \omega^{n-2} \leq c_1^2(E) \cdot \omega^{n-2}.$$

Now we can give the analogue of Theorem 1.18 \cite{DPS94} in the orbifold setting.
\begin{mythm}
Suppose that $X$ is a compact K\"ahler orbifold. Then an orbifold vector bundle
$E$ over $X$ is numerically flat (i.e. both $E, E^*$ are nef) if and only if $E$ admits a filtration
$$\{0\}= E_0  \subset E_1 \subset \cdots \subset E_p = E$$
by orbifold vector subbundles such that the quotients $E_k/E_{k-1}$ are Hermitian flat orbifold vector bundles.
\end{mythm}
\begin{proof}
The proof is analogous to those of \cite{CCM} and \cite{HIM}. 
The proof is obtained by considering the Harder-Narasimhan filtration of $E$ and showing that all the graded pieces are Hermitian flat.
For the convenience of the reader, we outline here the arguments with the necessary modifications. 

Consider the Harder-Narasimhan filtration of $E$ with respect to $\omega$, say
$$\cF_0=0 \to \cF_1 \to \cF_{2} \to \cdots\to \cF_m:=E$$
where $\cF_i /\cF_{i-1}$ is $\omega$-stable for every $i$ and $\mu_1 \geq \mu_2 \geq \cdots \geq \mu_m$, and where $\mu_j=\mu_\omega(\cF_j / \cF_{j-1})$ is the slope of $\cF_j / \cF_{j-1}$ with respect to $\omega$.
Now, consider the coherent orbifold subsheaf $\cS=\cF_{m-1}$.
Notice that by construction $\cS$ is reflexive by taking the double dual if necessary (on each local ramified smooth cover), as this preserves the rank, first Chern class and slope. Then we get a short exact sequence
$$0 \to \cS \to E \to \cQ \to 0.$$

In particular, its first Chern class $c_1(\cQ)$ is pseudo-effective by Lemma 4.  On the other hand, we have
$$0 =c_1(E) =c_1(\cS) +c_1(\cQ)$$
by the assumption. Thus
$$\int_X c_1(\cQ)\wedge \omega^{n-1}=-\int_X  c_1( \cS) \wedge \omega^{n-1} \leq 0,$$
and $c_1(\cQ)=c_1(\cS)=0$.

We claim that $\cS$ is an orbifold vector subbundle of $E$, and that the morphism $\cS \to E$ is an orbifold bundle morphism; for this, we apply Corollary~1.20 of \cite{DPS94} and prove that $\det(\cQ^*) \to \bigwedge^p E^*$ is an injective orbifold bundle morphism (which is local in nature), where $p$ is the rank of $\cQ$.
This corresponds to a global section $\tau \in H^0(X, \bigwedge^p E^* \otimes \det(\cQ^{**}))$.
Thus $ \tau$ cannot vanish at any point of $X$ by  Prop.~1.16 of \cite{DPS94} whose proof works identically in the orbifold setting.  
This concludes the proof of the claim.
In particular, $\cQ$ is an orbifold vector bundle.



Let $s$ be the rank of $\cS$, which must be strictly smaller than the rank $r$ of $E$. 
Since $E$ is nef and $\det\,\cQ^*$ is numerically trivial, we infer that $\cS$ is a nef orbifold vector bundle.

By the discussion of Segre forms, we have that
$$c_2(\cQ) \cdot \omega^{n-2} \leq c_1(\cQ)^2 \cdot \omega^{n-2} =0.$$
On the other hand, the Bogomolov inequality (Proposition 6) shows that 
$$\int_{X} c_2(\cQ) \wedge \omega^{n-2} \geq 0.$$
The orbifold vector bundle $\cQ$ is thus in fact a Hermitian flat orbifold vector bundle.

Notice that $\cS$ is also a numerically flat orbifold vector bundle. 
Applying inductively the above arguments concludes the proof. 
\end{proof}
The main technical lemma in \cite{Wu20} can also be generalised to the orbifold setting.
\begin{mythm}
Let $\pi: X \to Y$ be an orbifold submersion between compact K\"ahler orbifolds of relative dimension $r-1$. 
Let $T$ be a closed positive $(1,1)-$orbifold current in the cohomology class $\{ \alpha \} \in H^{1,1}(X, \R)$
such that $T$ has analytic singularities (meaning having analytic singularities in each local ramified smooth cover) and is orbifold smooth on $X \setminus \pi^{-1}(Z)$ with $Z$ a closed analytic set of codimension at least $k$.
Assume that for any $y \in Y$, there exist an open neighborhood $U$ of $y$ and a quasi-psh orbifold function $\psi$ on $X$ such that
$\alpha + i \d \dbar \psi \geq 0$
in the sense of currents on
$\pi^{-1}(U)$ and $\psi$ is orbifold smooth outside a closed analytic set of codimension at least $k+r$.
Then there exists a closed positive orbifold current in the cohomology class $\pi_* \alpha^{r+k-1}$.
\end{mythm}
\begin{proof}
The proof works almost identically other than one point.
In the proof of the uniqueness of weak limit, we use a resolution of singularities to reduce the analytic singularities to the divisorial case.
Since the uniqueness is local in nature, we can apply the same proof to some local ramified smooth cover (for which we apply the resolution of singularities).
\end{proof}
Using the Lelong number estimate and orbifold version of regularisation by smooth forms in the previous section, we have the following result.
\begin{mycor}
Let $E$ be a strongly psef orbifold vector bundle of rank $r$ on a compact K\"ahler orbifold $(X, \omega)$, such that $c_1(E)=0$. 
 Then $E$ is a nef $($and thus numerically flat$\,)$ orbifold vector bundle.
\end{mycor}
As a geometric application, we obtain the following generalisation of Theorem 7.7 in \cite{BDPP}.
\begin{mycor} {\it
For a compact K\"ahler orbifold, if $c_1(X) = 0$ and $T_X$ is strongly psef, then a finite quasi-\'etale cover of $X$
is a torus.
In particular, an irreducible symplectic, or Calabi-Yau orbifold does not have a strongly psef orbifold tangent bundle or orbifold cotangent bundle.}
\end{mycor}
\begin{proof}
By the orbifold version of Beauville-Bogomolov theorem given in \cite{Cam04}, up to a finite orbifold cover $\pi: \tilde{X} \to X$, $\tilde{X}$ is a product of $\prod T_i \times \prod S_j \times \prod Y_k$ where $T_i$ are complex tori, $S_j$ are Calabi-Yau orbifolds and $Y_k$ are irreducible symplectic orbifolds.
Since the tangent bundle of $\tilde{X}$ is numerically flat by Corollary 2, the orbifold tangent bundle of all the components in the direct sum is numerically flat.
In particular, all the components have vanishing second Chern class by Theorem 8.
By representation theory, the orbifold tangent bundle of the Calabi-Yau or irreducible symplectic components is stable (or by existance of orbifold K\"ahler-Einstein metric).
Thus we have the equality case in the Bogomolov inequality which implies that the orbifold tangent bundle of the Calabi-Yau or irreducible symplectic components is orbifold projectively flat.
Since the first Chern class of the Calabi-Yau or irreducible symplectic components vanishes,
the orbifold tangent bundle is in fact orbifold Hermitian flat.
In particular, the restricted holonomy groups of the Calabi-Yau or irreducible symplectic components are trivial.
In other words, there are only the complex tori components.
\end{proof}
By the work of \cite{LN19}, we can show in fact that an irreducible symplectic, or Calabi-Yau orbifold $X$ does not admit strongly psef orbifold vector bundle $\wedge^r T_X$ or $\wedge^r T_X^*$ for any $0<r<\mathrm{dim}_{\C} (X)$.
A stronger result in the projective singular setting can be found in Theorem 1.6 of \cite{HP}.

The following example shows that one possible generalisation of \cite{BDPP} to the category of projective orbifolds could not hold.
More precisely, there exists a rational projective variety $X$ with quotient singularities such that the orbifold canonical line bundle is pseudoeffective.
\begin{myex}{\em (log Enriques surfaces, \cite{Zha91}, \cite{Zha93})

The log Enriques surfaces are defined in \cite{Zha91} as follows.
A normal projective algebraic surface $X$ is said to be a log Enriques surface if
\begin{enumerate}
\item it has only quotient singularities and the singular part is non-empty;
\item $NK_X$ is a trivial Cartier divisor for some positive integer $N$;
\item $H^1(X,\cO_X)=0$. 
\end{enumerate}
Since $X$ is projective with quotient singularities, $X$ is a compact complex orbifold.
However, $X$ is not necessarily of canonical singularities.
Since $K_X$ is locally free on the regular part and coincides with the orbifold canonical line bundle, 
Condition 2 implies that the orbifold first Chern class of the orbifold canonical line bundle is trivial in $H^2(X,\R) \simeq H^2(X_\reg, \R)$ which is in particular pseudoeffective.

On the other hand, by Lemma 3.4 of \cite{Zha91}, if $K_X$ is not Cartier while $2K_X$ is Cartier, $X$ is rational.
We will always consider this special case from now on.
Notice that $X$ admits a quasi-\'etale cover $Y$ (called canonical covering in \cite{Zha91}) with $K_Y =\cO_Y$ which has at worst ordinary double point. 
The cover $Y$ gives the Beauville-Bogomolov decomposition given by \cite{HP} as explained below.
Consider $\pi: \tilde{Y} \to Y$ the minimal resolution of singularities of $Y$.
Since $Y$ has at worst ordinary double points, $K_{\tilde{Y}}$ is trivial.
In particular, $\tilde{Y}$ is either a torus or a K3 surface.
By Lemma 3.1 \cite{Zha91}, $\tilde{Y}$ is a K3 surface.
We show in the following that $Y$ is irreducible symplectic or irreducible Calabi-Yau in the sense of \cite{HP} which coincides with the case of surface.
Recall that the following definition is given in \cite{BGL}. Let $X$ be a normal compact K\"ahler complex space of dimension $n \geq 2$ with rational singularities. We call
$X$ irreducible holomorphic symplectic (IHS) if for all quasi-\'etale covers $q: X' \to X$, the algebra $H^0(X',\Omega^{[\bullet]}_{X'})$ is generated by a holomorphic symplectic form  $\tilde{\sigma} \in  H^0( X', \Omega^{[2]}_{X'})$.
We call $X$ irreducible Calabi–Yau (ICY) if for all quasi-\'etale covers $q: X' \to X$, the algebra $H^0(X',\Omega^{[\bullet]}_{X'})$ is generated by a nowhere vanishing reflexive form in degree
$\dim X$.
Recall that if $X$ is a normal complex space, the sheaf of reflexive p-forms $\Omega^{[p]}_X$ may be defined as the push-forward $j_* \Omega^p_{X_{\reg}}$ from the regular
locus $j: X_{\reg} \to X$ which is reflexive.


Now we show the condition on further quasi-\'etale cover.
Let $Z$ be a quasi-\'etale cover of $Y$. 
Then $K_Z$ is trivial since $Y$ is. Thus the minimal resolution of $Z$ is either a K3 surface or a torus.
In the second case, $Z$ itself is the minimal resolution since a torus does not contain any rational curve. 
We claim that $Z$ is not a torus.
Otherwise, by Lemma 2.8 of \cite{CGGN20}, let $Z'\to Z \to Y$ be the Galois cover of $Z \to Y$.
By construction, $Z' \to Z$ is quasi-\'etale.
Since $Z$ is smooth, $Z' \to Z$ is \'etale and $Z'$ is a torus.
In particular, $Y$ is a quotient of the torus by a finite group.
By Lemma 3.1 of \cite{Zha91} and the assumption that $Y$ is a quotient of the torus by a finite group, we can have only $A_1, A_3, A_5$ as possible singularity.
We denote the numbers of singular points to be $N_1, N_3,  N_5$ respectively.
Since the mininal resolution of $Y$ is a K3 surface, by holomorphic Lefschetz fixed-point formula,
we have
$$\chi(\cO_Y)=2=\frac{1}{12}(\frac{3}{2}N_1+\frac{15}{4}N_3+\frac{35}{6}N_5).$$
The only possibilities are $N_5=0, N_1=16, N_3=0$ or $N_5=0, N_1=11, N_3=2$ or $N_5=0, N_1=6, N_3=4$.
This contradicts Corollary 3.10 of \cite{Zha91}. 
}
\end{myex}
\section{An equivariant version of GAGA}
Let $X$ be a projective manifold with an (algebraic) action of a finite group $L$. 
Let $E$ be a holomorphic $L-$equivariant vector bundle over $X$.
Then by Serre's GAGA, the holomorphic morphism $L \times \P(E\oplus \cO_X) \to \P(E\oplus \cO_X)$ is the analytification of some algebraic morphism where $L$ acts on $\cO_X$ trivially.
In particular, $E$ is the analytification of some algebraic $L-$equivariant vector bundle.
For other kinds of $L$, we have the following equivariant version of GAGA communicated to us by Brion.
\begin{myprop}(Brion)

 Let $X$ be a projective manifold with an (algebraic) action of a connected reductive group $L$.
Let $E$ be a holomorphic $L-$equivariant vector bundle over $X$.
Then there exists a connected finite cover $L'$ of $L$ such that $E$ is an algebraic $L'-$equivariant vector bundle.
\end{myprop}
\begin{proof}
Recall that a holomorphic (resp. algebraic) equivariant vector bundle is a pair consisting of a holomorphic (resp. algebraic) vector bundle and the lifting of the holomorphic (resp. algebraic) action $L \times X \to X$ to that of $ L \times E \to E$ so that the natural projection $\pi: E \to X$ is equivariant. 
In particular, the action on $E$ is linear in each fibre which means that for any $t \in L$, any $x \in X$, the restriction of the action of $t$ on $E_x$ to $E_{t \cdot x}$ is a linear isomorphism.

Since $X$ is projective, by Serre's GAGA, the vector bundle $E$ is the analytification of some algebraic vector bundle which we still denote by $E$ with abuse of notation.
Let $G$ be the group of algebraic automorphisms of $E$ which is linear in each fibre.
One can show that
$G$ is a locally algebraic group (i.e. locally of finite type) by Proposition 6.3.2 of \cite{BSU13}.
There exists a natural morphism $\pi_*$ from $G$ to $\Aut(X)$.
(
For example, $\pi_*$ is induced by the restriction of the automorphism of $E$ to the zero section of $E$.)
By assumption, $t^* E$ is holomorphic isomorphic to $E$ for any $t\in L$.
Again by Serre's GAGA, this isomorphism is the analytification of some algebraic isomorphism.
In particular, for any $t \in L$, there exists an algebraic automorphism of $E$ which is linear in each fibre such that its image under $\pi_*$ is the automorphism of $X$ induced by the action of $t$ on $X$.
In other words, $L$ is contained in the image of $\pi_*$.

We claim that there exists a connected reductive algebraic subgroup $L'$ of $G$ such that the restriction $\pi_*$ on $L'$ is a finite cover of $L$ in our case.
Since $L'$ is a subgroup of $G$, it induces a lifting of the (algebraic) action $L' \times X \to X$ (naturally induced from $L' \to L$) to an algebraic action of $ L' \times E \to E$ so that the natural projection $\pi: E \to X$ is equivariant.

Now we prove the claim.
The kernel of $\pi_*$ is the group of automorphisms of the (algebraic) vector bundle $E$ (cf. Lemma 6.3.1 \cite{BSU13}).
Since $X$ is compact, the global sections of the endomorphism bundle of $E$ are of finite dimension.
These global sections form an associate algebra.
Let $e_i(1 \leq i \leq N)$ be the basis of these global sections. 
There exist $A_{ij}^k \in \C$ such that $e_i \circ e_j =\sum_k A_{ij}^k e_k$. 
Then $\Ker(\pi_*)$ is defined by $\det(\sum_i A_{ij}^k a_i)_j^k \neq 0$ in $\C^N=\{(a_1,\cdots,a_N)\} \simeq H^0(X,\End(E))$.
Thus $\Ker(\pi_*)$ is a connected affine variety.

Consider $H= \pi_*^{-1}(L) \subset G^{\circ}$ where $G^\circ$ is the component of $G$ containing the identity.
We have an exact sequence of algebraic groups 
$$0 \to \Ker(\pi_*) \to H \to L \to 0$$
which implies that $H \to L$ is a $\Ker(\pi_*)-$principal bundle.
In particular, $H \to L$ is an affine morphism.
Since $L$ is affine, $H$ is also affine.
Moreover $H$ is connected.
Thus $H$ is a connected linear algebraic group.

To continue the proof of the claim, we start the proof of the case that $L=(\C^*)^n$ for some $n \in \N$ (i.e. $L$ is an algebraic torus).
In this case, a stronger result than our Proposition is proven in \cite{Ste21} if $X$ is toric and $E$ is a toric vector bundle:
the analytification gives an equivalence of categories between algebraic toric vector bundles on $X$ and holomorphic toric vector bundles
on $X^{\mathrm{an}}$.
Notice that here we do not assume that $E$ is toric (as the action on $X$ is not necessarily the toric action).

Let $T_H$ be a maxiaml torus of $H$.
By Proposition 11.14 of \cite{Bo91}, the restriction of $H \to L$ on $T_H$ is also surjective.
By Corollary 8.3 of \cite{Bo91}, we have a bijective contravariant correspondence between algebraic tori and free abelian groups of finite rank by associating an algebraic torus with its character group.
The surjection $T_H \to L$ corresponds to the inclusion of character group $i: X^*(L) \to X^{*}(T_H)$.
Let $f_\alpha (1 \leq \alpha \leq r)$ be a basis of $X^*(L)$.
Complete $(i(f_\alpha))_\alpha$ by elements $g_\beta(1 \leq \beta \leq s)$ of $X^{*}(T_H)$ such that $i(f_\alpha), g_\beta$ form a $\Q-$linear basis of $X^{*}(T_H) \otimes_\Z \Q$.
Let $\Lambda$ be the abelian group generated by $i(f_\alpha), g_\beta$ as a standard basis.
There exists $N \in \N^*$ such that $X^{*}(T_H) \subset \frac{1}{N} \Lambda \subset X^{*}(T_H) \otimes_\Z \Q$.
Let $X^*(L')$ be the abelian group generated by $N$ times the first $r$ elements in the standard basis of $\Lambda$ with natural map $X^*(T_H) \to X^*(L')$ induced from the projection on the first $r$ components.
Then the composition map of $X^*(L) \to X^*(T_H) \to X^*(L')$ has finite cokernel.
This corresponds to an algebraic subtorus $L'$ of $T_H$ such that $L' \to T_H \to L$ is finite.

Now we turn to the proof of the general case.
By 11.22 of \cite{Bo91}, there exists a Levi subgroup $M$ of $H$ which is a connected subgroup such that
$H$ is the semi-direct product of $M$ and $R_u(H)$ the unipotent radical of $H$.
Since $L$ is reductive, the image of $R_u(H)$ is contained in $R_u(L)$ which is trivial.
In particular, there exists a surjective morphism of connected reductive groups $M \simeq H/R_u(H) \to L$ which induces a surjection of Lie algebras.

Let $T_L$ (resp. $T_M$) be a maximal torus of $L$ (resp. $M$) such that $T_M$ maps surjectively onto $T_L$.
We have the decomposition of Lie algebras in eigenspaces via the adjoint representation
$$\Lie(M)=\Lie(T_M) \oplus \oplus_{\alpha \in \Phi(M)} \g_{\alpha,M}$$
where $\Lie()$ means the corresponding Lie algebra of the group and $\g_{\alpha,M}:=\{\eta \in \g | \mathrm{Ad}(t) \cdot \eta = \alpha(t) \eta, \forall t \in  T_M\}$ and $\Phi(M)$ are the roots.
We have a similar decomposition for $L$.
The surjection of Lie algebras implies that for any $\alpha \in \Phi(L)$, there exists some $\phi(\alpha) \in \Phi(M)$ such that the restriction $\Lie(M) \to \Lie(L)$ on $\g_{\phi(\alpha)}$ is an isomorphism onto $\g_\alpha$.
By the structure of connected reductive Lie group (cf. e.g. \cite{Sp98}), there exists a unique closed connected unipotent
subgroup $U_{\phi(\alpha)} \subset M$ normalised by $T_M$ such that $\Lie(U_{\phi(\alpha)})=\g_{\phi(\alpha)}$.
Let $T'_M$ be an algebraic subtorus of $T_M$ constructed in the previous case such that $T'_M \to T_L$ is finite.
Define $L'$ to be the subgroup of $M$ generated by $T'_M$ and all the $U_{\phi(\alpha)}$'s for $\alpha \in \Phi(L)$.
By Proposition 2.2, Chapter 1 \cite{Bo91}, the Lie group $L'$ is also closed.
The surjection of $\Lie(L') \to \Lie(L)$ implies that $L' \to L$ is surjective.
By dimension reason, $L' \to L$ is finite.
The Lie group $L'$ is connected since all groups generating it are connected.

In the end, we can also argue as follows. 
By 14.2 \cite{Bo91}, $L=C(L) \cdot [L,L]$ where $C(L)$ is the maximal algebraic torus contained in the center of $L$ and $C(L)\cap [L,L]$ is finite.
Note that $[L,L]$ is semisimple and $L/C(L) \simeq [L,L]$.
The surjection $M \to L$ implies the surjections $C(M) \to C(L)$, $[M,M] \to [L,L]$.
Let $C'(M)$ be an algebraic subtorus of $C(M)$ constructed in the previous case such that $C'(M) \to C(L)$ is finite.
Let $ \Lie([M,M]) \to \Lie([L,L])$ be the induced surjection of Lie algebras.
Decompose $\Lie([M,M]), \Lie([L,L])$ into direct sum of simple Lie algebras.
By the definition of simple Lie algebra, any simple piece of $\Lie([M,M])$ either maps isomorphically into some simple piece of $\Lie([L,L])$ or maps to 0.
Thus there exists a sub Lie algebra of $\Lie([M,M])$ such that the restriction of $\Lie([M,M]) \to \Lie([L,L])$ on it is an isomorphism.
Consider the corresponding Lie subgroup of $[M,M]$ of this sub Lie algebra and consider the closed Lie subgroup of $M$ generated by this Lie subgroup of $[M,M]$ and $C'(M)$.
This gives the connected closed reductive subgroup $L'$.
\end{proof}
Notice that the proof in fact shows a result slightly stronger:
Let $X$ be a projective manifold with an (algebraic) action of a connected reductive group $L$.
Let $E$ be a holomorphic $L-$invariant vector bundle over $X$ (i.e. for any $t \in L$, $t^* L \simeq L$).
Then there exists a connected finite cover $L'$ of $L$ such that $E$ is an algebraic $L'-$equivariant vector bundle.
\begin{myrem}
{\em 
Notice that the construction of $L'$ a priori depends on the vector bundle $E$.
If the rank of $E$ is 1, $L'$ can be chosen to be independent of $E$.
By Theorem 5.2.1 of \cite{Bri18}, $E$ is always $L$-invariant in this case. Moreover, there exists a positive integer $n$ such that $E^{\otimes n}$ is $L$-equivariant; we may take for $n$ the exponent of the finite abelian Picard group of $L$.
In our case, let $\widetilde{[L,L]}$ be the universal cover of $[L,L]$.
Consider $L':= C(L) \times \widetilde{[L,L]}$.
Since the Picard group of an algebra torus or a simply connected semisimple Lie group is trivial,
the Picard group of $L'$ is trivial.
Since $C(L)$ is contained in the center of $L$, the natural morphisms $C(L) \to L$, $\widetilde{[L,L]} \to [L,L] \to L$ induces a finite cover $L' \to L$.
In particular, for any line bundle $E$, $E$ is always $L'-$equivariant by Theorem 5.2.1 of \cite{Bri18}.
Notice that $E$ is not necessarily $L-$equivariant.
Such an example is constructed for example in Example 4.2.4 of \cite{Bri18}.
Notice that Brion's example is neither holomorphically $L-$equivariant nor algebraically $L-$equivariant.
It seems to be unclear in general whether a holomorphic $L-$equivariant vector bundle over a projective manifold is necessarily algebraically $L-$equivariant or not.
}
\end{myrem}
\begin{myrem}
{\em 
Our result can be reformulated as follows.
Let $X$ be a projective manifold with an (algebraic) action of a connected reductive group $L$.
Consider the following two categories.
The elements of the categories are the holomorphic (resp. algebraic) $L'-$equivariant vector bundles over $X$ where $L'$ is some connected finite cover of $L$.
The morphisms between an $L_1-$equivariant bundle $E_1$ and an $L_2-$equivariant bundle $E_2$ are the holomorphic (resp. algebraic) $L_3-$equivariant bundle morphisms where $L_3$ is some common finite cover of $L_1$ and $L_2$ with natural induced actions on $E_1$ and $E_2$.
Consider the functor of analytification.
Then our result states that this functor is an equivalence of categories.
In fact, Proposition 7 implies that the analytification functor is essentially surjective.
By Serre's GAGA, the analytification functor is full which is easy to check that it is faithful.

}
\end{myrem}
In the end, we give a characterisation of projective toric variety.
\begin{myprop}
Let $X$ be a connected compact K\"ahler manifold.
Let $\Aut^\circ(X)$ be the connected component of the automorphism group.
Let $G \subset  \Aut^\circ(X)$ be a connected reductive group.
Assume that there exists $z \in X$ such that $G \cdot z$ is open and dense (i.e. $X$ is $G-$quasi homogeneous).
Assume that $G$ acts on $\Alb(X)$ trivially.

Then $X$ is projective.
\end{myprop}
\begin{proof}
Let $F (X)$ be
the irreducible component of the Douady space of $X \times X$ containing the diagonal.
Let $F(X)_\red$ be the reduced reduction of $F(X)$.
By the fundamental work of Fujiki \cite{Fu78}, 
$F(X)_\red$ is a meromorphic
structure on $\Aut^\circ(X)$.
More precisely,  $F(X)_\red$ is a compact complex space containing $\Aut^\circ(X)$ 
as a dense open subset such that the product
and the inversion of $\Aut^\circ(X)$ extend as meromorphic maps on $F(X)_\red$.

By Theorem 4.6 \cite{BG19}, the group $G$ acts on $\Alb(X)$ trivially if and only if the closure of $G$ in $F(X)_\red$ is compact and contains $G$ has a dense open set such that the product
and the inversion of $G$ extend as meromorphic maps on the closure.

Let $B(X)$ be the closure of $\{\Gamma_f,f \in \Aut^\circ(X)\} \simeq \Aut^\circ(X)$ the graphs $\Gamma_f$ of $f$ in the  Barlet cycle space of $X \times X$.
Then $B(X)$ is an irreducible component of the cycle space of $X \times X$ and $B(X) \setminus B^\circ(X)$ is closed analytic (cf. Proposition 3.1 \cite{BG19}).

By Theorem 4.11 \cite{BG19}, the closure of $G$ in $B(X)$ is a projective variety.
The restriction of the natural morphism from Douady space to cycle space is a modification of the closure of $G$ in $B(X)$.
In particular, the closure $\overline{G}$ of $G$ in $F(X)_\red$ has a smooth projective bimeromorphic model.

By Proposition 2.2 \cite{Fu78}, the action $G \times X \to X$ extends to a meromorphic map on the closure of $G$ in $F(X)_\red$ times $X$.
Consider the projection of the intersection of its graph $\Gamma_{\overline{G} \times X \to X}$ with $\overline{G} \times \{z\}$ (which is projective) into $X$.
The image of the projection contains the orbit $G \cdot z$.
By Remmert's proper mapping theorem, the image is a closed analytic subset of $X$.
Thus by our assumption, the image is $X$.
Thus $X$ is Moishezon as the image of surjective morphism from a projective variety.
Since $X$ is also K\"ahler,
$X$ is projective.
\end{proof}
\begin{myrem}
{\em
The statement is false for an arbitrary compact complex manifold.
For example, Calabi–Eckmann manifold $X$ is a $\GL(m,\C) \times \GL(n,\C)-$homogeneous (for some $m,n \geq 2$) manifold.
If $m \neq n$, the $m-$th and $n-$th Betti number is 1,
thus it is not a manifold of class $\mathcal{C}$.
The action of $\Aut^\circ(X)$ on the Dolbeault cohomologies is trivial as shown in Lemma 9.3 \cite{Hir66} by Borel.

It is asked in Remark 4.12 \cite{BG19} whether their results hold for a compact complex manifold of class $\mathcal{C}$.
If the response is positive, the previous proposition can also be generalised to a compact complex manifold of class $\mathcal{C}$ with the conclusion that $X$ is Moishezon.


}
\end{myrem}    
\begin{mycor}
Let $X$ be a connected compact K\"ahler manifold which is $(\C^*)^n-$quasi-homogeneous with $n=\dim_\C(X)$ such that the action on the open dense set is the multiplication of $(\C^*)^n$.
Then $X$ is a projective toric variety.
\end{mycor}
\begin{proof}
To apply the previous Proposition 8, it is enough to show that $(\C^*)^n$ acts on $\Alb(X)$ trivially. 
Note that a posterior, any projective smooth toric variety is simply connected (Theorem 12.1.10 \cite{CLS}).
In particular, the first Betti number is 0 and by Hodge decomposition, the Albanese torus is trivial.

Let $U=(\C^*)^n \cdot x$ be the open dense orbit and $\alpha:  X \to \Alb(X)$ be the Albanese map.
Since $X$ is also $\Aut^\circ(X)-$quasi-homogeneous, the Albanese map of $X$ is locally trivial by Remark 4 \cite{Wu21b}.
In particular, $\alpha^{-1}(\alpha(x))$ is a connected compact complex manifold.
Thus $\alpha^{-1}(\alpha(x)) \cap (X \setminus U)$ is a proper closed analytic subset of $\alpha^{-1}(\alpha(x))$.
In particular, $\alpha^{-1}(\alpha(x)) \cap U$ is connected.

Assume by contrary that the action on $\Alb(X)$ is not trivial.
In particular, the Albanese torus has a positive dimension.
Consider the composition map $f:(\C^*)^n \to X \to \Alb(X) $ which is $(\C^*)^n-$equivariant. 
Thus $f$ is not a constant map.
Then $f^{-1}(f(1))$ is a proper closed subgroup of $(\C^*)^n$.
Every closed subgroup of $\C^*$ is either trivial, $\C^*$ or discrete. 
By connectedness, $f^{-1}(f(1))$ is isomorphic to $(\C^*)^d$ for some $d$.
By density, $d$ is equal to the dimension of the fibres of the Albanese map.

The induced action of $(\C^*)^n/ f^{-1}(f(1))$ on $\Alb(X)$ is effective,
thus the induced group morphism $(\C^*)^n/ f^{-1}(f(1)) \to \Aut^\circ(\Alb(X)) $ is injective.
Since both sides are connected and have the same dimension, this is an isomorphism.
Both sides are thus trivial.
Contradiction with the dimension of $\Alb(X)$.
\end{proof}
Note that by a theorem of Sumihiro \cite{Sum74}, for a connected projective manifold which is $(\C^*)^n-$quasi homogeneous with $n=\dim_\C(X)$ such that the action on the open dense set is the multiplication of $(\C^*)^n$,
it is necessarily constructed from a fan (cf. Corollary 3.8 \cite{CLS}).


\begin{thebibliography}{9}
\bibitem[BCGP12]{BCGP}
Ingrid Bauer, Fabrizio Catanese, Fritz Grunewald and Roberto Pignatelli,{\em
Quotients of products of curves, new surfaces with $p_g=0$ and their fundamental groups.}
 Amer. J. Math. {\bf{134}} (2012), no. 4, 993–1049. 
\bibitem[BDPP13]{BDPP}
S\'ebastien Boucksom, Jean-Pierre Demailly, Mihai P$\check{a}$un and Thomas Peternell, 
{\em The pseudo-effective cone of a
compact Kähler manifold and varieties of negative Kodaira dimension, } arXiv:math/0405285,
J. Algebraic Geom. {\bf{22}},(2013) no. 2, 201-248.
\bibitem[BEG13]{BEG}
S\'ebastien Boucksom, Philippe Eyssidieux and Vincent Guedj, {\em
An introduction to the Kähler-Ricci flow,} Lecture Notes in Mathematics, 2086. Springer, Cham, 2013. viii+333 pp. 
\bibitem[BG19]{BG19}
Leonardo Biliotti, Alessandro Ghigi, {\em
Meromorphic limits of automorphisms,}
arXiv:1901.10724, Transform. Groups {\bf{26}} (2021), no. 4, 1147–1168. 
\bibitem[BGL20]{BGL}
Benjamin Bakker, Henri Guenancia, Christian Lehn, 
{\em
Algebraic approximation and the decomposition theorem for Kähler Calabi-Yau varieties,}
arXiv:2012.00441.  Invent. Math. {\bf{228}} (2022), no. 3, 1255–1308. 
\bibitem[BHPV04]{BHPV04}
Wolf Barth, Klaus Hulek, Chris A.M. Peters, Antonius Van de Ven, 
\textit{Compact complex surfaces},  
A Series of Modern Surveys in Mathematics, {\bf{4}}. Springer-Verlag, Berlin, 2004. xii+436 pp. ISBN: 3-540-00832-2. 
\bibitem[Bin83]{Bin83}
Jürgen Bingener, {\em
On Deformations of Kähler Spaces. I.}
 Math. Z. {\bf{182}} (1983), no. 4, 505–535. 
\bibitem[BMT86]{BMT86}
Francesco Guaraldo, Patrizia Macrì, Alessandro Tancredi,
{\em 
Topics on Real Analytic Spaces,}
Advanced Lectures in Mathematics. Friedr. Vieweg and Sohn, Braunschweig, 1986. x+163 pp. ISBN: 3-528-08963-6 
\bibitem[Bo91]{Bo91}
Armand Borel,
{\em Linear Algebraic Groups,}
Second edition. Graduate Texts in Mathematics, 126. Springer-Verlag, New York, 1991. xii+288 pp. ISBN: 0-387-97370-2
\bibitem[Bog78]{Bog78}
Fedor Alekseevich Bogomolov, {\em
Holomorphic tensors and vector bundles on projective manifolds. } 
Izv. Akad. Nauk SSSR Ser. Mat.{\bf{ 42}} (1978), no. 6, 1227–1287, 1439.

 \bibitem[Bri18]{Bri18}
Michel Brion,
{\em Linearization of algebraic group actions},
 Handbook of group actions. Vol. IV, 291–340, Adv. Lect. Math. (ALM), {\bf{41}}, Int. Press, Somerville, MA, 2018. 
\bibitem[BSU13]{BSU13}
Michel Brion, Preena Samuel, V. Uma,
{\em Lectures on the structure of algebraic groups and geometric applications,}
CMI Lecture Series in Mathematics, 1. Hindustan Book Agency, New Delhi; Chennai Mathematical Institute (CMI), Chennai, 2013. viii+120 pp. ISBN: 978-93-80250-46-5 
\bibitem[Cam04]{Cam04}
 Frédéric Campana, {\em Orbifoldes à première classe de Chern nulle.} (French) [Orbifolds of zero first Chern class] The Fano Conference, 339–351, Univ. Torino, Turin, 2004. 
\bibitem[Car57]{Car57}
Henri Cartan,
{\em Quotient d'un espace analytique par un groupe d'automorphismes.} A symposium in honor of S. Lefschetz, Algebraic geometry and topology. pp. 90–102. Princeton University Press, Princeton, N. J. 1957.
\bibitem[CCM21]{CCM}
Frédéric Campana, Junyan Cao, Shin-ichi Matsumura,
{\em Projective klt pairs with nef anti-canonical divisor.}
arXiv:1910.06471,
 Algebr. Geom.{\bf{ 8}} (2021), no. 4, 430–464. 
\bibitem[CGGN22]{CGGN20}
Benoît Claudon, Patrick Graf, Henri Guenancia, Philipp Naumann,
{\em
Kähler spaces with zero first Chern class: Bochner principle, fundamental groups, and the Kodaira problem,}
arXiv:2008.13008.
 J. Reine Angew. Math.{\bf{ 786}} (2022), 245–275. 
\bibitem[Che55]{Che}
Claude Chevalley, {\em Invariants of finite groups generated by reflections.} Amer. J. Math.{\bf{ 77}} (1955), 778–782.
\bibitem[CLS11]{CLS}
David A. Cox, John B. Little, Henry K. Schenck, 
{\em Toric Varieties,}
Graduate Studies in Mathematics, 124. American Mathematical Society, Providence, RI, 2011. xxiv+841 pp. ISBN: 978-0-8218-4819-7 
\bibitem[Col90]{Col90}
Mihnea Coltoiu, {\em Complete locally pluripolar sets.}  J. Reine Angew. Math. {\bf{ 412}} (1990), 108–112.
\bibitem[CR02]{CR02}
Weimin Chen, Yongbin Ruan, {\em Orbifold Gromov-Witten theory.} Orbifolds in mathematics and physics (Madison, WI, 2001), 25–85, Contemp. Math., {\bf{ 310}}, Amer. Math. Soc., Providence, RI, 2002. 
\bibitem[Dem85]{Dem85}
Jean-Pierre Demailly,
{\em 
Mesures de Monge-Ampère et caractérisation géométrique des variétés algébriques affines,}  Mém. Soc. Math. France (N.S.) No. 19 (1985), 124 pp. 
\bibitem[Dem92]{Dem92}
Jean-Pierre Demailly, {\em
Regularization of closed positive currents and intersection theory.}  J. Algebraic Geom. {\bf{ 1}} (1992), no. 3, 361–409.
\bibitem[Dem94]{Dem94}
Jean-Pierre Demailly, {\em
Regularization of closed positive currents of type (1,1) by the flow of a Chern connection,}  Contributions to complex analysis and analytic geometry, 105–126,
Aspects Math., E26, Friedr. Vieweg, Braunschweig, 1994. 
\bibitem[Dem12]{agbook}
Jean-Pierre Demailly, 
\newblock{\em Complex analytic and differential geometry.} online-book: https://www-fourier.ujf-grenoble.fr/~demailly/manuscripts/agbook.pdf, (2012).
\bibitem[Dem22]{Dem21}
Jean-Pierre Demailly, {\em
Bergman bundles and applications to the geometry of compact complex manifolds,}  Pure Appl. Math. Q. {\bf{ 18}} (2022), no. 1, 211–249. 
\bibitem[DHP08]{DHP}
Jean-Pierre Demailly, Jun-Muk Hwang and Thomas Peternell, {\em Compact Manifolds covered by a torus,}  J. Geom. Anal. {\bf{ 18}} (2008), no. 2, 324–340. 
\bibitem[Div16]{Div16}
Simone Diverio. {\em Segre forms and Kobayashi-Lübke inequality.}  Math. Z. {\bf{ 283}} (2016), no. 3-4, 1033–1047. 
\bibitem[DPS94]{DPS94}
Jean-Pierre Demailly, Thomas Peternell, and Michael Schneider,
{\em Compact complex manifolds with numerically effective tangent bundles,}  J. Algebraic Geom.{\bf{ 3}} (1994), no. 2, 295–345.
\bibitem[DPS01]{DPS01}
Jean-Pierre Demailly, Thomas Peternell, and Michael Schneider, 
{\em Pseudo-effective line bundles on compact K\"ahler manifolds,}  Internat. J. Math. {\bf{ 12}} (2001), no. 6, 689–741. 
\bibitem[Dol82]{Dol82}
Igor Dolgachev, {\em
Weighted projective varieties,}
 Group actions and vector fields (Vancouver, B.C., 1981), 34–71, Lecture Notes in Math., 956, Springer, Berlin, 1982.
\bibitem[Fau22]{Fau19}
Mitchell Faulk, {\em
Hermitian-Einstein metrics on stable vector bundles over compact Kähler orbifolds,}
arXiv:2202.08885
\bibitem[Fis76]{Fis76}
Gabriele Fischer, {\em
 Complex Analytic Geometry,} Lecture Notes in Mathematics, Vol. 538. Springer-Verlag, Berlin-New York, 1976. vii+201 pp. 
  \bibitem[FN80]{FN}
John Erik Fornaess, Raghavan Narasimhan,  {\em The Levi problem on complex spaces with
singularities.}  Math. Ann. {\bf{ 248}} (1980), no. 1, 47–72. 
\bibitem[Fu78]{Fu78}
Akira Fujiki. {\em On automorphism groups of compact K\"ahler manifolds.}  Invent. Math. {\bf{ 44}} (1978), no. 3, 225–258. 
\bibitem[GH10]{GH10}
Daniel Greb, Peter Heinzner, {\em
 Kählerian Reduction in Steps,}
Symmetry and spaces, 63–82,
Progr. Math., 278, Birkhäuser Boston, Boston, MA, 2010. 
 \bibitem[GK20]{GK20}
Patrick Graf, Tim Kirschner, 
{\em Finite quotients of three-dimensional complex tori,} 
 Ann. Inst. Fourier (Grenoble) {\bf{ 70}} (2020), no. 2, 881–914. 
\bibitem[GKKP11]{GKKP11}
Daniel Greb, Stefan Kebekus, Sándor J. Kovács, Thomas Peternell, {\em Differential forms on log canonical spaces.}  Publ. Math. Inst. Hautes Études Sci. No. 114 (2011), 87–169.
\bibitem[GR70]{GR}
Hans Grauert, Oswald Riemenschneider,
{\em  Verschwindungss\"atze f\"ur analytische 
Kohomologiegruppen auf komplexen R\"aumen, }
 Invent. Math. {\bf{ 11}} (1970), 263–292.
\bibitem[Gul12]{Gul12}
Dincer Guler.{\em On Segre forms of positive vector bundles.}
 Canad. Math. Bull. {\bf{ 55}} (2012), no. 1, 108–113. 
\bibitem[GW75]{GW75}
Robert Everist Greene, Hung Hsi Wu, {\em Embedding of open riemannian manifolds by
harmonic functions,}  Ann. Inst. Fourier (Grenoble) {\bf{ 25}} (1975), no. 1, vii, 215–235. 
\bibitem[Hat02]{Hat}
Allen Hatcher.
{\em Algebraic Topology, }
 Cambridge University Press, Cambridge, 2002. xii+544 pp. ISBN: 0-521-79160-X; 0-521-79540-0 
\bibitem[Her67]{Her67}
Miguel Herrera, 
{\em De Rham theorems on semianalytic sets.}  Bull. Amer. Math. Soc. {\bf{ 73}} (1967), 414–418.
\bibitem[HIM22]{HIM}
Genki Hosono, Masataka Iwai, Shin-ichi Matsumura,
{\em On projective manifolds with pseudo-effective tangent bundle.}
 J. Inst. Math. Jussieu {\bf{ 21}} (2022), no. 5, 1801–1830.
\bibitem[Hir66]{Hir66}
Friedrich Hirzebruch, {\em Topological methods in algebraic geometry,} Third enlarged edition. New appendix
and translation from the second German edition by R. L. E. Schwarzenberger, with an additional
section by A. Borel. Die Grundlehren der Mathematischen Wissenschaften, Band 131, Springer-
Verlag New York, Inc., New York, 1966. MR 0202713 (34 2573).
\bibitem[HP19]{HP}
Andreas H\"oring, Thomas Peternell.
{\em Algebraic integrability of foliations with numerically trivial canonical bundle.} Invent. Math. {\bf{216}}, (2019), 395–419. https://doi.org/10.1007/s00222-018-00853-2
\bibitem[Kan13]{Kan13}
Marja Kankaanrinta, {\em
On real analytic orbifolds and Riemannian metrics,}
 Algebr. Geom. Topol.{\bf{ 13}} (2013), no. 4, 2369–2381. 
\bibitem[Kis79]{Kis79}
Christer Oscar Kiselman, {\em Densit\'e des fonctions plurisousharmoniques,}  Bull. Soc. Math. France  {\bf{107}} (1979), no. 3, 295–304. 
\bibitem[LT18]{LT16}
Steven Lu, Behrouz Taji, {\em A characterization of finite quotients of abelian varieties,}
 Int. Math. Res. Not. IMRN 2018, no. 1, 292–319. 
\bibitem[LN19]{LN19}
Fatima Laytimi, Werner Nahm,
{\em 
Ampleness equivalence and dominance for vector bundles,}
 Geom. Dedicata {\bf{ 200}} (2019), 77–84. 
\bibitem[MM07]{MM}
Xiaonan Ma, George Marinescu,
\textit{Holomorphic Morse inequalities and Bergman kernels,}
Progress in Mathematics, 254. 
Birkh\"auser Verlag, Basel, 2007. xiv+422 pp. ISBN: 978-3-7643-8096-0.  
\bibitem[MM10]{MM10}
Ieke Moerdijk, Janez Mrcun, {\em
Introduction to Foliations and Lie Groupoids,}
Cambridge Studies in Advanced Mathematics, 91. Cambridge University Press, Cambridge, 2003. x+173 pp. ISBN: 0-521-83197-0 
\bibitem[Mou04]{Mou04}
Christophe Mourougane. 
{\em Computations of Bott-Chern classes on $\P(E)$.}  Duke Math. J. {\bf{ 124}} (2004), no. 2, 389–420. 
\bibitem[Pet59]{Pet59}
Franklin P. Peterson,
{\em Some Remarks on Chern Classes,}
 Ann. of Math. (2) {\bf{69}} (1959), 414–420.
\bibitem[Qui76]{Qui76}
Daniel Quillen, {\em Projective modules over polynomial rings},  Invent. Math.{\bf{ 36}} (1976), 167–171.
\bibitem[Ric68]{Ric68}
Rolf Richberg, {\em Stetige streng pseudokonvexe Funktionen,}  Math. Ann. {\bf{175}} (1968), 257–286. 
\bibitem[Rie71]{Rie}
 Oswald Riemenschneider, 
 {\em Characterizing Moishezon Spaces by Almost Positive Coherent Analytic Sheaves, }
 Math. Z. {\bf{123}} (1971), 263–284.
\bibitem[Ros68]{Ros}
Hugo Rossi,
{\em  Picard variety of an isolated singular point,}
 Rice Univ. Stud. {\bf{54}} (1968), no. 4, 63–73. 
\bibitem[Pri67]{Pri67}
David Prill, {\em
Local classification of quotients of complex manifolds
by discontinuous groups,}
 Duke Math. J. {\bf{34}} (1967), 375–386. 
\bibitem[PS03]{PS03}
Chris A. M. Peters, Joseph H. M. Steenbrink, 
{\em
Degeneration of the Leray Spectral Sequence for Certain Geometric Quotients,}  Mosc. Math. J. {\bf{3}} (2003), no. 3, 1085–1095, 1201. 
\bibitem[Sum74]{Sum74}
Hideyasu Sumihiro, {\em Equivariant completion, I, II, } J. Math. Kyoto Univ. {\bf{14}} (1974), 1–28;{\bf{ 15}} (1975),
573–605.
\bibitem[Sp98]{Sp98}
Tonny Albert Springer,
{\em Linear Algebraic Groups,} Second edition. Progress in Mathematics, 9. Birkhäuser Boston, Inc., Boston, MA, 1998. xiv+334 pp. ISBN: 0-8176-4021-5
\bibitem[ST54]{ST}
Geoffrey C. Shephard, John Arthur Todd, 
{\em Finite unitary reflection groups.}  Canad. J. Math. {\bf{ 6}} (1954), 274–304. 
\bibitem[Ste21]{Ste21}
Jonas Stelzig,
{\em Toric vector bundles: GAGA and Hodge theory,}
 Math. Z. {\bf{ 298}} (2021), no. 1-2, 771–786. 
\bibitem[Var89]{Var}
Jean Varouchas,
{\em
Kähler spaces and proper open morphisms,}
 Math. Ann. {\bf{283}} (1989), no. 1, 13–52. 
\bibitem[Vie77]{Vie77}
Eckart Viehweg. Rational singularities of higher dimensional schemes. Proc. Amer. Math. Soc.,
63(1):6–8, 1977.
\bibitem[Voi07]{Voi}
Claire Voisin, 
{\em Hodge theory and complex algebraic geometry I.}
 Cambridge Studies in Advanced Mathematics, 76. Cambridge University Press, Cambridge, 2007. x+322 pp. ISBN: 978-0-521-71801-1
 \bibitem[Wu22]{Wu20}
 Xiaojun Wu, \textit{Strongly pseudo-effective and numerically flat reflexive sheaves}, 
J. Geom. Anal. {\bf{32}} (2022), no. 4, Paper No. 124, 61 pp. 
 \bibitem[Wu21]{Wu21}
 Xiaojun Wu, {\em
 The Bogomolov's inequality on a singular complex space,}
 arXiv:2106.14650.
 \bibitem[Wu21b]{Wu21b}
Xiaojun Wu,
{\em Note on compact Kähler manifold with strongly
pseudo-effective tangent bundle,}
arXiv:2110.02931.
\bibitem[Zha91]{Zha91}
 De-Qi Zhang, {\em Logarithmic Enriques surfaces,}  J. Math. Kyoto Univ. {\bf{ 31}} (1991), no. 2, 419–466.
 \bibitem[Zha93]{Zha93}
 De-Qi Zhang, {\em
 Logarithmic Enriques surfaces II,}
 J. Math. Kyoto Univ. {\bf{33}} (1993), no. 2, 357–397.
\end{thebibliography}
\end{document}